\documentclass[12pt]{article}

\usepackage{amsmath}
\usepackage{amssymb}
\usepackage{amsthm}
\usepackage{amsfonts}
\usepackage{amscd}
\usepackage{graphicx}
\usepackage[T1]{fontenc}    
\usepackage[nottoc]{tocbibind}  
\usepackage[usenames,dvipsnames]{color}
\usepackage{enumitem} 
\usepackage{caption}
\usepackage[numbers,sort&compress]{natbib}
\usepackage{hyperref}
\usepackage{adjustbox}
\usepackage{fancyvrb} 
\usepackage{tikz}
\usepackage{color}
\usepackage{verbatim} 
\usepackage{framed}
\usepackage{geometry}
\usepackage[medium]{titlesec}
\usepackage{etoolbox} 
\usepackage{arydshln} 
\usepackage[capitalise]{cleveref} 
\usepackage{tocloft} 
\usepackage{appendix}

\newtheorem{theorem}{Theorem}
\newtheorem{corollary}{Corollary}

\newtheorem{lemma}{Lemma}
\newtheorem{proposition}[lemma]{Proposition}

\newtheorem{theoremext}{Theorem}

\theoremstyle{definition}

\newtheorem{example}[lemma]{Example}
\newtheorem{remark}[lemma]{Remark}

\definecolor{colLinkBlue}{RGB}{23,111,192} 
\definecolor{colCiteGreen}{RGB}{8,144,8} 
\definecolor{colP}{RGB}{170,0,255} 
\definecolor{colLP}{RGB}{255,170,255} 
\definecolor{colBrown}{RGB}{156,99,49} 
\definecolor{colGray}{RGB}{128,128,128} 
\definecolor{colGreen}{RGB}{0,204,0} 
\definecolor{colO}{RGB}{255,170,0} 
\definecolor{colLB}{RGB}{143,189,211} 
\definecolor{colB}{HTML}{6699CC} 
\definecolor{colR}{HTML}{CC6677} 

\hypersetup{colorlinks,urlcolor=colCiteGreen,citecolor=colCiteGreen,linkcolor=colLinkBlue}
\titlespacing*{\section}{0pt}{0mm}{0mm}
\titlespacing*{\subsection}{0pt}{0mm}{0mm}
\titlespacing*{\paragraph}{0pt}{0mm}{0mm}
\newcommand{\myspace}{\setlength{\abovedisplayskip}{1mm}\setlength{\belowdisplayskip}{0mm}}
\usetikzlibrary{cd}
\usetikzlibrary{calc}
\tikzstyle{vertex}=[circle, draw=white, fill=black!60, minimum size=4pt, inner sep=0pt]
\tikzstyle{edge}=[black!30,line width=1.75pt]
\tikzstyle{labelsty}=[font=\scriptsize]
\tikzstyle{traj}=[line width=1pt, draw=colR,densely dashed]

\newenvironment{Mlist}{\begin{itemize}[topsep=0pt,itemsep=0pt,leftmargin=7mm]}{\end{itemize}}
\newenvironment{Menum}{\begin{enumerate}[topsep=0pt,itemsep=0pt,leftmargin=10mm,label=(\roman*),ref=(\roman*)]}{\end{enumerate}}
\newenvironment{claims}{\begin{enumerate}[topsep=0pt,itemsep=0pt,leftmargin=8mm,label=\textnormal{\textbf{(\alph*)}},ref=(\alph*)]}{\end{enumerate}}

\newcommand{\csep}[1]{\setlength{\tabcolsep}{#1}}
\newcommand{\fig}[3]{\includegraphics[height=#1cm, width=#2cm]{img/#3}}

\newcommand{\ASN}[1]{Assertion~\ref{#1}}

\newcommand{\APP}[1]{Appendix~\ref{sec:#1}}
\newcommand{\EQN}[1]{(\ref{eqn:#1})}
\newcommand{\AXM}[1]{\ref{axm:#1}}
\newcommand{\SEC}[1]{\textsection\ref{sec:#1}}
\newcommand{\RL}[2]{\Cref{lem:#1}\ref{lem:#1:#2}}
\newcommand{\RLS}[4]{Lemmas~\ref{lem:#1}\ref{lem:#1:#2} and~\ref{lem:#3}\ref{lem:#3:#4}}
\newcommand{\RP}[2]{\Cref{prp:#1}\ref{prp:#1:#2}}
\newcommand{\END}{\hfill $\vartriangleleft$}

\newcommand{\df}[1]{{\it #1}}
\newcommand{\Wlog}{without loss of generality }
\newcommand{\resp}{respectively}
\newcommand{\st}{such that }
\newcommand{\wrt}{with respect to }

\renewcommand{\c}{\colon}
\newcommand{\set}[2]{\left\{#1:#2\right\}}
\newcommand{\dto}{\dasharrow}

\newcommand{\codim}{\operatorname{codim}}
\newcommand{\Sing}{\operatorname{Sing}}

\newcommand{\C}{{\mathbb{C}}}
\newcommand{\R}{{\mathbb{R}}}
\newcommand{\Z}{{\mathbb{Z}}}

\newcommand{\F}{{\mathbb{F}}}
\newcommand{\G}{{\mathbb{G}}}
\newcommand{\E}{{\mathbb{E}}}

\newcommand{\cE}{\mathcal{E}}
\newcommand{\cP}{\mathcal{P}}
\newcommand{\cC}{\mathcal{C}}
\newcommand{\cS}{\mathcal{S}}

\newcommand{\bt}{\mathbf{t}}
\newcommand{\bu}{\mathbf{u}}

\newcommand{\tr}{\tilde{r}}
\newcommand{\ts}{\tilde{s}}
\newcommand{\tbt}{\tilde{\bt}}
\newcommand{\tbu}{\tilde{\bu}}

\newcommand{\m}{\mathfrak{m}}

\newcommand{\SE}{\mathrm{SE}(2)}
\newcommand{\sd}{d} 
\newcommand{\CM}{\cC_{_{G,\m}}}
\newcommand{\CouplerCurve}{\operatorname{CouplerCurve}}
\newcommand{\Spec}{\operatorname{Spec}}
\newcommand{\lc}{\operatorname{lc}}

\newcommand{\fT}{\mathfrak{T}}
\newcommand{\zv}{\mathbf{0}}

\newif\ifhide
\ifhide
\renewenvironment{tikzpicture}{D\comment}{\endcomment}
\renewcommand{\fig}[3]{F}
\fi

\begin{document}
\myspace

\begin{center}
{\LARGE
Irreducible components of sets of points in the plane that satisfy distance conditions
}
\\[5mm]
{\large
Niels Lubbes,
Mehdi Makhul,
Josef Schicho,
Audie Warren
}
\\[5mm]
{\large\today}
\end{center}

\begin{abstract}
For a given graph whose edges are labeled with general real numbers,
we consider the set of functions from the vertex set
into the Euclidean plane such that
the distance between the
images of neighbouring vertices is equal to the corresponding edge label.
This set of functions can be expressed as the zero set of quadratic polynomials
and our main result characterizes the number of complex irreducible components
of this zero set
in terms of combinatorial properties of the graph.
In case the complex components are three-dimensional,
then the graph is minimally rigid and the component number is a well-known invariant from rigidity theory.
If the components are four-dimensional, then they correspond to one-dimensional coupler curves of flexible planar mechanisms.
As an application, we characterize the degree of irreducible components of such coupler curves combinatorially.


\textbf{MSC Class:}
52C25, 
70B15, 
51K05, 
51F99  
\end{abstract}

\section{Introduction}
\label{sec:intro}

In \Cref{fig:mr}, we see an example of a graph with vertex set~$\{1,2,3,4\}$
and each edge~$\{i,j\}$ has some ``edge length assignment''~$\lambda_{\{i,j\}}$.
We assign to each vertex~$i$ a point~$(x_i,y_i)$ in the plane \st
\[
(x_i-x_j)^2+(y_i-y_j)^2=\lambda_{\{i,j\}}.
\]
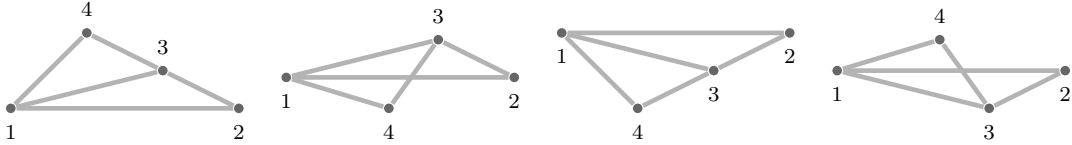
\begin{figure}[!ht]
\centering
\csep{1mm}
\begin{tabular}{cccc}
\begin{tikzpicture}[scale=1]
\node[vertex] (a) at (0,0)   [label={[labelsty]below:$1$}]{};
\node[vertex] (b) at (3,0)   [label={[labelsty]below:$2$}]{};
\node[vertex] (c) at (2,0.5) [label={[labelsty]above:$3$}]{};
\node[vertex] (d) at (1,1)   [label={[labelsty]above:$4$}]{};
\draw[edge] (a)edge(c) (b)edge(c) (c)edge(d) (a)edge(d) (a)edge(b);
\end{tikzpicture}
&
\begin{tikzpicture}[scale=1]
\node[vertex] (a) at (0,0)        [label={[labelsty]below:$1$}]{};
\node[vertex] (b) at (3,0)        [label={[labelsty]below:$2$}]{};
\node[vertex] (c) at (2,0.5)      [label={[labelsty]above:$3$}]{};
\node[vertex] (d) at (1.35,-0.41) [label={[labelsty]below:$4$}]{};
\draw[edge] (a)edge(c) (b)edge(c) (c)edge(d) (a)edge(d) (a)edge(b);
\end{tikzpicture}
&
\begin{tikzpicture}[scale=1]
\node[vertex] (a) at (0,0)    [label={[labelsty]below:$1$}]{};
\node[vertex] (b) at (3,0)    [label={[labelsty]below:$2$}]{};
\node[vertex] (c) at (2,-0.5) [label={[labelsty]below:$3$}]{};
\node[vertex] (d) at (1,-1)   [label={[labelsty]below:$4$}]{};
\draw[edge] (a)edge(c) (b)edge(c) (c)edge(d) (a)edge(d) (a)edge(b);
\end{tikzpicture}
&
\begin{tikzpicture}[scale=1]
\node[vertex] (a) at (0,0)       [label={[labelsty]below:$1$}]{};
\node[vertex] (b) at (3,0)       [label={[labelsty]below:$2$}]{};
\node[vertex] (c) at (2,-0.5)    [label={[labelsty]below:$3$}]{};
\node[vertex] (d) at (1.35,0.41) [label={[labelsty]above:$4$}]{};
\draw[edge] (a)edge(c) (b)edge(c) (c)edge(d) (a)edge(d) (a)edge(b);
\end{tikzpicture}
\end{tabular}
\caption{
A $\lambda$-compatible realization of the graph
is up to translations and rotations equivalent to one of four representatives.}
\label{fig:mr}
\end{figure}
The complex zero set of these quadratic polynomials consists
of four 3-dimensional components. Indeed,
if we additionally set $(x_1,y_1)=(0,0)$ and $(x_2,y_2)=(1,0)$,
then this system of quadratic equations has four solutions.
In other words, the graph admits up to rotations and translations
four ``$\lambda$-compatible realizations'' into the plane.
We shall make these concepts precise in \SEC{main}, but for now notice that
a representative of each component is illustrated in \Cref{fig:mr}.

The graph in \Cref{fig:mr} is an example of a ``minimally rigid graph''
whose ``number of realizations'' equals four.
Such minimally rigid graphs
can be characterized in terms of the number of vertices and edges of subgraphs \cite{1927,1970}.
The number of realizations of a minimally rigid graph with $n$ vertices is, over the complex numbers,
the same for almost all edge length assignments, and is therefore a graph invariant~\cite{2019}.
This number is at most $\binom{2n-4}{n-2}\approx 4^n$ \cite{2004}
and can be recovered from the graph in terms of a purely combinatorial algorithm~\cite{2018-num}.
However, the number of \emph{real} realizations (drawings in~$\R^2$)
does depend on the choice of edge length assignment and remains an open problem \citep[\textsection8]{2019}.

If we remove an edge from a minimally rigid graph
with general choice of edge length assignment~$\lambda$,
then we expect that the resulting graph admits up to translations and rotations a
1-dimensional choice of $\lambda$-compatible realizations.
In \Cref{fig:coupler}, we illustrate
two realizations of such a graph, where the edge $\{1,0\}$ was removed.
\begin{figure}[!hb]
\centering
\csep{10mm}
\begin{tabular}{cc}
\begin{tikzpicture}
\node[inner sep=0pt] at (1.6,0) {\includegraphics{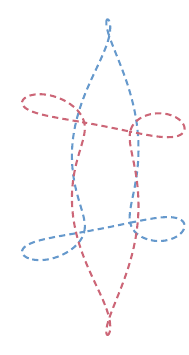}};
\node[vertex] (a) at (0,0)   [label={[labelsty]below:$1$}] {};
\node[vertex] (b) at (3,0)   [label={[labelsty]below:$2$}] {};
\node[vertex] (c) at (1,2)   [label={[labelsty]above:$3$}] {};
\node[vertex] (d) at (2,1.5) [label={[labelsty]below:$4$}] {};
\node[vertex] (e) at (1,1)   [label={[labelsty]below:$0$}] {};
\draw[edge] (a)edge(b) (b)edge(d) (d)edge(c) (c)edge(a) (c)edge(e) (e)edge(d);
\end{tikzpicture}
&
\begin{tikzpicture}
\node[inner sep=0pt] at (1.6,0) {\includegraphics{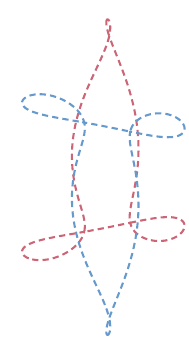}};
\node[vertex] (a) at (0,0)     [label={[labelsty]below:$1$}] {};
\node[vertex] (b) at (3,0)     [label={[labelsty]below:$2$}] {};
\node[vertex] (c) at (1,2)     [label={[labelsty]above:$3$}] {};
\node[vertex] (d) at (2,1.5)   [label={[labelsty]below:$4$}] {};
\node[vertex] (e) at (1.8,2.6) [label={[labelsty]above:$0$}] {};
\draw[edge] (a)edge(b) (b)edge(d) (d)edge(c) (c)edge(a) (c)edge(e) (e)edge(d);
\end{tikzpicture}
\end{tabular}
\caption{The 1-dimensional set of realizations consists of two irreducible components
that are isomorphic to the red and blue components of the coupler curve of the vertex~$0$.}
\label{fig:coupler}
\end{figure}
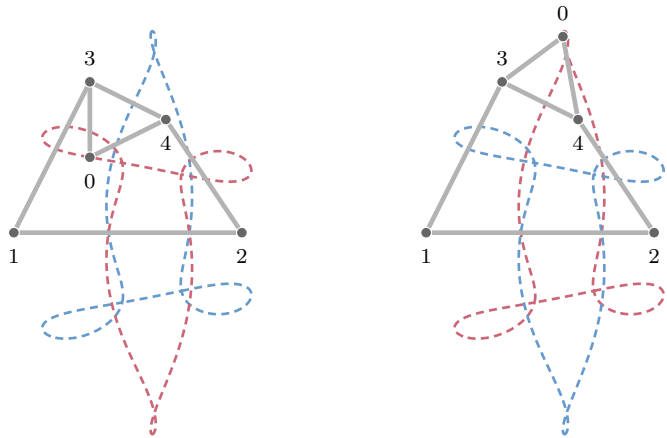

The blue and red curves represent all realizations of the vertex~$0$
with the conditions $(x_1,y_1)=(0,0)$ and $(x_2,y_2)=(1,0)$.
In this example, almost all realizations of the vertices $\{0,1,2\}$
uniquely determines the realizations of the remaining vertices.
Therefore, the set of all $\lambda$-compatible realizations has two irreducible components.
The blue and red curves are projections of these two components onto the realizations of the vertex~$0$.

We recover this number of components in a combinatorial manner as follows.
First, we decompose the graph into a union of minimally rigid subgraphs that are not contained in
a bigger minimally rigid subgraph. In \Cref{fig:coupler}, these ``max-tight subgraphs''
are the triangle~$\{0,3,4\}$, and the edges $\{1,2\}$, $\{1,3\}$ and~$\{2,4\}$.
The number of realizations of the triangle is $2$
and each of the three edges has realization number~$1$.
Our main result, namely \Cref{thm:main}, now implies that the number
of irreducible components of the set of $\lambda$-compatible realizations
is equal to the product of the number of realizations of the
max-tight subgraphs: $2\cdot 1\cdot 1\cdot 1=2$.
Note that in this article, we count the components over the complex numbers.
The number of real connected components remains an open problem (see \Cref{rmk:real}).

We may think of the vertices in \Cref{fig:coupler} as revolute joints and after we pin down the vertices~$1$ and~$2$,
the graph moves as the vertex~$0$ traverses the red ``coupler curve'' component.
A key observation is that
the corresponding realizations of the triangle~$\{0,3,4\}$
are all equivalent up to rotations and translations.
\Cref{cor:main} implies that the converse holds as well:
two $\lambda$-compatible realizations
belong to the same component if and only if each of the max-tight subgraphs
are equivalent up to rotations and translations.
\Cref{cor:cal} characterizes the degrees
of the components of coupler curves and addresses \citep[Conjecture~2 and Open problem~3]{2023-cal}.

If we remove more than one edge from some minimally rigid graph,
then the irreducible components
of the set of compatible realizations has higher dimension.
\Cref{cor:inv} shows that
the degree and geometric genera of each component has to be the same.

\subsection*{State of the art}

The investigation of rigid structures can be traced back to
James Clerk Maxwell~\cite{1864} and there is currently again a considerable interest in rigidity
theory due its applications in
robotics, architecture, molecular biology and sensor network localization.
We refer to \cite{2002} and \citep[Introduction]{2022-ub} for further references to such applications.

Realizations of minimally rigid graphs have been studied
in Euclidean spaces and spheres of arbitrary dimension.
Theoretical and algorithmic aspects of the number of realizations
have been investigated in \cite{2021-iso,2019,2018-num,2023-cal,2023-sphere}.
Of particular interest are the globally rigid graphs, namely graphs
that up to rotations, translations and reflections admit a unique realization \cite{2005,2023-glob}.
Bounds on the number of realizations have been investigated in
\cite{2004,2022-ub,2021-real,2020-low,2010,2023-dewar}.
We remark that for carefully chosen edge lengths minimally rigid
graphs may admit infinitely many realizations~\cite{2022-yet}.

Graphs that admit a 1-dimensional set of
realization for some choice of edge lengths
are special cases of mechanical linkages with revolute joints.
Their origins can be traced back to at least 1785,
when James Watt used such a linkage for his steam engine.
Moving graphs remain of interest to engineers
and their coupler curves have been studied extensively \cite{1990}.
In \cite{2023-cal} we provide bounds on degrees and geometric genera
of coupler curves of mechanical linkages in terms of numbers of realizations
(see also \cite{1963} and \citep[Figure~5]{2004}).
We improve these bounds in \Cref{cor:cal}.

\subsection*{Overview}

In \SEC{main}, we present the assertions of the main result \Cref{thm:main}
and \Cref{cor:main}.
We show in \SEC{pair} that \Cref{thm:main} is implied
by the irreducibility of certain fiber products associated to graphs
that have not too many edges (the proof of a key lemma is deferred to \APP{fiber}).
In \SEC{red}, we prove \Cref{thm:main,cor:main}
by showing that these fiber products are irreducible
via induction on the number of vertices.
This concludes the main content of this article.
In the remaining \SEC{app} and \APP{scheme},
we present applications of our main results, namely \Cref{cor:inv,cor:cal}.
In particular, we combinatorially characterize
the number of irreducible components of coupler curves of planar mechanisms
and their degrees.
In \SEC{app} is also a short discussion on the real connected components
of sets of realizations.
We conclude in \SEC{dim} with some remarks on the higher dimensional analogs of our main results.

\section{Statements of the main results}
\label{sec:main}

In this section, we state our main results \Cref{thm:main,cor:main},
whose proof will be concluded in \SEC{red}.
For an overview of the proof, see the introductions of \SEC{pair} and \SEC{red}.

\subsection*{Terminology needed for main results}
Suppose that $P\c X\to \{\text{True},~\text{False}\}$ is a predicate about some variety~$X$.
We say that $P$ holds for a \df{general} point if
the set $\set{x\in X}{P(x)=\text{False}}$ is contained in a lower dimensional subvariety of~$X$.

The \df{squared distance function}~$d\c\C^2\to \C$ is defined as
\[
\sd(x,y)=(x_1-y_1)^2+(x_2-y_2)^2.
\]
The \df{transformation group}~$T$ is defined as the set of maps
\[
\C^2\to\C^2,\quad (x_1,x_2) \mapsto (c\,x_1-s\,x_2+u,~ s\,x_1+c\,x_2+v)
\]
\st $c,s,u,v\in\C$ and $c^2+s^2=1$.
We refer to elements of $T$ as \df{transformations.}
Notice that transformations \st $c,s,u,v\in\R$ define
the subgroup $\SE\subset T$ of directed isometries of $\R^2$.
Indeed, for all $t\in T$,
\[
\sd(x,y)=\sd(t(x),t(y)).
\]
In this article a \df{graph} $G=(V,E)$ with \df{vertex set}~$V$ and \df{edge set}~$E$
is a simple undirected finite graph.
We call $G'\subset G$ a subgraph if $V'\subset V$ and $E'\subset E$,
where the subset operators $\subset$ and $\subseteq$ are considered the same.

If $X$ and $Y$ are sets, then $Y^X$ denotes the set of functions $X\to Y$ and
for all $f\in Y^X$ and $x\in X$, we denote the evaluation~$f(x)$ by $f_x$.
By abuse of notation, we denote
\[
\C^{2V}:=(\C^2)^{V}.
\]
We shall refer to elements of $\C^{2V}$ and $\C^E$ as \df{realizations} and \df{edge length assignments}, \resp.
The \df{edge map} of a graph $G=(V,E)$ is defined as
\[
\cE_G\c \C^{2V}\to\C^E,\quad r\mapsto \bigl(\{u,v\}\mapsto d(r_u,r_v)\bigr).
\]
The \df{set of $\lambda$-compatible realizations} of $G$ for $\lambda\in \C^E$ is defined as the fiber
\[
\cE_G^{-1}(\lambda)=\set{r\in \C^{2V}}{\sd(r_u,r_v)=\lambda(\{u,v\}) \text{ for all }\{u,v\}\in E}.
\]
If $E=\varnothing$, then $\cE_G^{-1}(\C^\varnothing)=\C^{2V}$.
The \df{component number} $c(G,\lambda)$ for $G$ \wrt $\lambda\in\C^E$
is defined as the number of irreducible components of~$\cE_G^{-1}(\lambda)$.

We denote by $\cE_G^{-1}(\lambda)/T$ the orbit set, where
the transformation group~$T$ acts on
the set of realizations~$\C^{2V}$ as follows:
\[
T\times \C^{2V}\to \C^{2V},\quad (t,r)\mapsto t\circ r.
\]
We call a graph $G$ \df{minimally rigid} if there exists nonzero~$n\in\Z_{> 0}$
\st for general~$\lambda\in \C^E$,
\[
c(G):=|\cE_G^{-1}(\lambda)/T|=n.
\]
We refer to $c(G)$ as the \df{number of realizations} of~$G$.
We remark that if $G$ is minimally rigid, then for general $\lambda\in\C^E$
each irreducible component of $\cE_G^{-1}(\lambda)$ is a $T$-orbit and
thus
\[
c(G)=c(G,\lambda).
\]
Following the dictionary in \cite{2008}, we call a graph $G$ \df{sparse}
if all its subgraphs~$(V', E')$ \st $|V'|\geq 2$ satisfy
\[
|E'| \leq 2|V'| - 3.
\]
We call a graph $G=(V,E)$ \df{tight} if it is sparse and either $|V|=1$ or $|E|=2|V| - 3$.

The following theorem of Pollaczek-Geiringer \cite{1927} and independently Laman \cite{1970}
characterizes minimally rigid graphs in combinatorial terms.

\begin{theoremext}[Geiringer, Laman]
\label{thm:laman}
~\\
A graph $G$ is \df{minimally rigid} if and only if $G$ is tight.
\end{theoremext}

We call a subgraph $G'\subset G$ \df{max-tight} if it is tight and there exists
no tight subgraph~$G''$ \st $G'\subsetneq G''\subset G$.

We call $\{G_i\}_{i\in I}=\{(V_i,E_i)\}_{i\in I}$ a \df{max-tight decomposition} for a graph~$G=(V,E)$
if $G_i\subset G$ is a max-tight subgraph for all $i\in I$,
$V=\cup_{i\in I}V_i$, $E=\cup_{i\in I} E_i$
and $E_i\cap E_j=\varnothing$ for all different $i,j\in I$.

See \cite{2025pyrigi} for a
\href{https://pyrigi.github.io/PyRigi/userguide/api/graph.html#pyrigi.graph.Graph.rigid_components}{software implementation}
of an algorithm that for a given graph outputs its max-tight decomposition.

\subsection*{Main results and examples}

Our main result is that the component number of a sparse graph
\wrt a general edge length assignment
is equal to the product of the numbers of realizations of its max-tight subgraphs.

\begin{theorem}
\label{thm:main}
If $G=(V,E)$ is a sparse graph and $\lambda\in\C^E$ a general edge length assignment,
then there exists a unique max-tight decomposition $\{G_i\}_{i\in I}$ for $G$ and
\[
c(G,\lambda)=\prod_{i\in I} c(G_i).
\]
\end{theorem}

We remark that if $G$ is not sparse and $\lambda\in\C^E$ a general edge length assignment,
then $c(G,\lambda)=0$ as a straightforward consequence of \Cref{thm:laman}.

If $\lambda$ is a general edge length assignment of a graph~$G$,
then any pair of $\lambda$-compatible realizations of $G$
belong to the same component of the set of $\lambda$-compatible realizations
if and only if
for each max-tight $G'\subset G$ the induced pair of realizations of $G'$
are equal up to composition with transformations.

\begin{corollary}
\label{cor:main}
If $G=(V,E)$ is a sparse graph, $\{G_i\}_{i\in I}=\{(V_i,E_i)\}_{i\in I}$ its max-tight decomposition
and $\lambda\in\C^E$ a general edge length assignment,
then the following are equivalent for all $\lambda$-compatible realizations $r$ and $s$ in $\cE_G^{-1}(\lambda)$:
\begin{Menum}
\item\label{cor:main:i}
The realizations $r$ and $s$ belong to the same irreducible component of~$\cE_G^{-1}(\lambda)$.
\item\label{cor:main:ii}
For all $i\in I$ there exists a transformation $t\in T$ \st $t\circ r|_{V_i}=s|_{V_i}$.
\item\label{cor:main:iii}
For all $i\in I$, the restricted realizations $r|_{V_i}$ and $s|_{V_i}$
belong to the same irreducible component of~$\cE_{G_i}^{-1}(\lambda|_{E_i})$.
\end{Menum}
\end{corollary}

We give some examples to explain the main theorem.

\begin{example}
\label{exm:main1}
Suppose that $G$ is the graph in \Cref{fig:main1}.
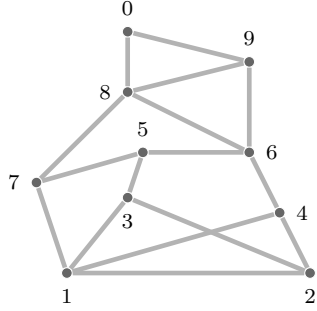
\begin{figure}[!ht]
\centering
\begin{tikzpicture}[scale=0.4]
\node[vertex] (a) at ( 0  ,0  )  [label={[labelsty]below:$1$}] {};
\node[vertex] (b) at ( 8  ,0  )  [label={[labelsty]below:$2$}] {};
\node[vertex] (c) at ( 2  ,2.5)  [label={[labelsty]below:$3$}] {};
\node[vertex] (d) at ( 7  ,2  )  [label={[labelsty]right:$4$}] {};
\node[vertex] (e) at ( 2.5,4  )  [label={[labelsty]above:$5$}] {};
\node[vertex] (f) at ( 6  ,4  )  [label={[labelsty]right:$6$}] {};
\node[vertex] (g) at (-1  ,3  )  [label={[labelsty]left :$7$}] {};
\node[vertex] (h) at ( 2  ,6  )  [label={[labelsty]left :$8$}] {};
\node[vertex] (i) at ( 6  ,7  )  [label={[labelsty]above:$9$}] {};
\node[vertex] (j) at ( 2  ,8  )  [label={[labelsty]above:$0$}] {};
\draw[edge] (a)edge(b) (a)edge(c) (a)edge(d) (b)edge(d) (b)edge(c)
            (a)edge(g) (g)edge(e) (e)edge(c) (e)edge(f) (f)edge(d)
            (g)edge(h) (h)edge(f) (h)edge(i) (i)edge(f) (h)edge(j) (j)edge(i);
\end{tikzpicture}
\caption{A sparse graph with $8$ max-tight subgraphs and component number~$16$.}
\label{fig:main1}
\end{figure}

Let us determine its component number~$c(G,\lambda)$ for general
edge length assignment~$\lambda$.
We observe that $G$ is sparse.
Let $G_1$ and $G_2$ be the max-tight subgraphs that are spanned by the vertices $\{1,2,3,4\}$
and $\{0,8,9,6\}$, \resp.
The remaining six max-tight subgraphs $G_3,\ldots,G_8$ of $G$ correspond to the
edges~$\{1,7\}$, $\{7,8\}$, $\{7,5\}$, $\{5,3\}$, $\{5,6\}$ and~$\{6,4\}$.
Hence, $\{G_1,\ldots,G_8\}$ is a max-tight decomposition for~$G$.
Recall from \Cref{fig:mr} that the numbers of realizations~$c(G_1)$ and~$c(G_2)$
are both equal to~$4$.
For the remaining subgraphs, we have $c(G_1)=\ldots=c(G_2)=1$.
We now conclude from \Cref{thm:main} that $c(G,\lambda)=c(G_1)\cdots c(G_8)=16$.
\END
\end{example}

\begin{example}
\label{exm:main2}
Let $G$ be the graph in \Cref{fig:main2},
which is obtained by removing the edge $\{1,0\}$ from the complete bipartite
$K_{3,3}$ graph.

\begin{figure}[!ht]
\centering
\begin{tikzpicture}[scale=0.4]
\node[vertex] (a) at ( 1,0)  [label={[labelsty]below:$1$}] {};
\node[vertex] (b) at ( 7,0)  [label={[labelsty]below:$2$}] {};
\node[vertex] (c) at ( 2,3)  [label={[labelsty]above:$5$}] {};
\node[vertex] (d) at ( 6,3)  [label={[labelsty]above:$0$}] {};
\node[vertex] (e) at (-1,5)  [label={[labelsty]above:$3$}] {};
\node[vertex] (f) at ( 9,5)  [label={[labelsty]above:$4$}] {};
\draw[edge] (a)edge(b) (a)edge(e)
            (c)edge(b) (c)edge(d) (c)edge(e)
            (f)edge(b) (f)edge(d) (f)edge(e);
\end{tikzpicture}
\caption{A sparse graph with component number $1$.}
\label{fig:main2}
\end{figure}
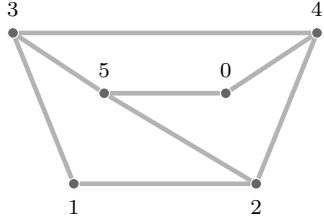

Each max-tight subgraph of $G$ consists of an edge and thus its component number~$c(G,\lambda)$
for general edge length assignment~$\lambda$ is equal to~$1$.
\END
\end{example}

\section{Pairs of realizations as a fiber product}
\label{sec:pair}

In this section, we construct for a sparse graph~$G=(V,E)$ a certain fiber product~$P(G)$ whose elements
correspond to pairs of realizations of~$G$ \st
for each max-tight subgraph of~$G$
their induced pair of realizations are equal up to composition with transformations.

For example, if $G$ is as in \Cref{fig:coupler},
then we require for a real element of $P(G)$
that the induced pair of realizations of
the max-tight triangle~$\{0,3,4\}$
are related by a directed isometry (and thus a transformation).
Notice that the two realizations depicted in \Cref{fig:coupler}
do not define an element of~$P(G)$ since the triangles are related by
a reflection (which is not a transformation).

Suppose we are given a realization of the graph~$G$ together with
a transformation assigned to each of its max-tight subgraphs.
The ``subgraph map''~$\cS_G$ assigns to this data
a compatible realization of each max-tight subgraph.
For example, the two realizations in \Cref{fig:coupler}
belong roughly speaking to different fibers of~$\cS_G$.
The set $P(G)$ is contained in the fiber product associated to the map~$\cS_G$.

The edge map~$\cE_G$ fits together with the subgraph map~$\cS_G$ into a certain commutative diagram
and the number of components of the general fiber of~$\cE_G$
is equal to the product of these ``fiber component numbers'' for $\cS_G$
and all the edge maps associated to the max-tight subgraphs of~$G$.
Our main result \Cref{thm:main} is equivalent to the fiber component
number of $\cS_G$ to be equal to one which in turn is equivalent to
the fiber product~$P(G)$ being irreducible.
The irreducibility of $P(G)$ will be established in the next section~\SEC{red}.

Let us proceed with making the concepts in the above section overview precise.

Suppose that $X$ and $Y$ are (not necessarily irreducible) varieties.
We call a rational map~$f\c X\dto Y$ between varieties
a \df{morphism} if it is everywhere defined and write $f\c X\to Y$.
We call a morphism~$f$ \df{dominant} if its image~$f(X)$ is Zariski dense in $Y$.
If $Y$ is irreducible, then we define the \df{fiber component number}~$c(f)$
as the number of irreducible components of $f^{-1}(y)$ for general~$y\in Y$.

\begin{lemma}
\label{lem:em}
Suppose that $G=(V,E)$ is a graph.
\begin{claims}
\item\label{lem:em:a}
The graph $G$ is sparse if and only if its edge map $\cE_G$ is dominant.
\item\label{lem:em:b}
If $G$ is sparse, then for general edge length assignment $\lambda\in\C^E$,
\[
c(G,\lambda)=c(\cE_G).
\]
If $G$ is also tight, then $c(G)=c(\cE_G)$.
\item\label{lem:em:c}
If $G$ is sparse, then it admits a unique max-tight decomposition
$\{(V_i,E_i)\}_{i\in I}$.
\end{claims}
\end{lemma}

\begin{proof}
\ref{lem:em:a}
By definition, a sparse graph is either tight or obtained by removing edges from some tight graph
and thus this assertion is a direct consequence of \Cref{thm:laman}.

\ref{lem:em:b}
Since $\cE_G$ is dominant by \ASN{lem:em:a} and $\C^E$ is irreducible,
the fiber component number of $\cE_G$ is well-defined as it is the number of components of the ``generic fiber''
(see \citep[Definition~1.5 in Chapter~3]{2002L}).

\ref{lem:em:c}
This assertion follows from the fact that if two tight graphs have a common edge,
then their union is again tight.
\end{proof}

A \df{tight decomposition} for a graph $G=(V,E)$
is defined as
\[
\Omega:=\{G_i\}_{i\in I}=\{(V_i,E_i)\}_{i\in I}
\]
\st
$G_i\subset G$ is a tight subgraph for all $i\in I$,
$V=\cup_{i\in I}V_i$, $E=\cup_{i\in I}E_i$
and $E_i\cap E_j=\varnothing$ for all different $i,j\in I$.

The \df{subgraph map} associated to the tight decomposition~$\Omega$ is defined as
\[
\cS_\Omega\c \C^{2V}\times T^I\to \prod_{i\in I}\C^{2V_i},\quad (r,\bt)\mapsto (\bt_i\circ r|_{V_i})_{i\in I}.
\]
If $G$ is sparse and $\Omega$ its max-tight decomposition, then we write
\[
\cS_G:=\cS_\Omega.
\]
Indeed, the max-tight decomposition of a sparse graph is unique by \RL{em}{c}.

A subgraph map fits together with an edge map into the following diagram:
\begin{equation}
\label{eqn:dia}
\begin{tikzcd}
\C^{2V}\times T^I\arrow[d, "\pi",swap]\arrow[rr, "\cS_\Omega"] & & \prod_{i\in I}\C^{2V_i}\arrow[d,"\cP_G=\prod_{i\in I}\cE_{G_i}"]\\
\C^{2V}\arrow[r,"\cE_G"] & \arrow[r,"\mu","\cong"']\C^E & \prod_{i\in I}\C^{E_i}
\end{tikzcd}
\end{equation}
where
$\cP_G\bigl((r^i)_{i\in I}\bigr)=\bigl(\cE_{G_i}(r^i)\bigr)_{i\in I}$ is a product of edge maps,
$\mu(\lambda)=(\lambda|_{E_i})_{i\in I}$ is an isomorphism
and
$\pi(r,\bt)=r$ is the projection to the first component.

\begin{lemma}
\label{lem:dia}
If $G=(V,E)$ is a sparse graph with tight decomposition
\[
\Omega=\{G_i\}_{i\in I}=\{(V_i,E_i)\}_{i\in I},
\]
then the diagram at \EQN{dia} commutes.
\end{lemma}

\begin{proof}
Since $E_i\cap E_j=\varnothing$ for all different $i,j\in I$,
we find that $\mu$ is an isomorphism.
Applying transformations to realizations of each of the tight subgraphs of $G$
does not change the squared distances.
We deduce that the diagram commutes.
\end{proof}

\begin{lemma}
\label{lem:fc}
Suppose that $f\c X\to Y$ is a continuous map between the topological spaces $X$ and~$Y$
and let $\G$ be a connected topological group.
If $f$ is $\G$-equivariant \wrt continuous and fixed-point free group actions
\[
\cdot\c \G\times X\to X
\quad\text{and}\quad
\cdot\c \G\times Y\to Y,
\]
then the number of connected components of $f^{-1}(y)$ is equal to
the number of connected components of $f^{-1}(\set{g\cdot y}{g\in \G})$ for all $y\in Y$.
\end{lemma}

\begin{proof}
Let $\G\cdot y:=\set{g\cdot y}{g\in \G}$ and let $\{C_i\}_{i\in I}$, $\{B_j\}_{j\in J}$
be pairwise different connected components \st
\[
f^{-1}(y)=\cup_{i\in I} C_i
\qquad\text{and}\qquad
f^{-1}(\G\cdot y)=\cup_{j\in J} B_j.
\]
We need to show that $|I|=|J|$.
Notice that $f^{-1}(y)\subset f^{-1}(\G\cdot y)$ and thus for all~$i\in I$ there exists $j\in J$
\st
\[
C_i\subset f^{-1}(y)\cap B_j.
\]
Now suppose by contradiction that there exists different $a,b\in I$ and $k\in J$
\st $C_a,C_b\subset B_k$.
In this case, there exists a continuous path
$p\c [0,1]\to B_k$ \st $p(0)\in C_a$ and $p(1)\in C_b$.
By assumption $f(B_k)\subset \G\cdot y$ and $\G$ acts fixed-point free on~$Y$.
Hence, there exists a unique continuous path $g\c [0,1]\to \G$ \st
\begin{Mlist}
\item $(f\circ p)(t)=g(t)\cdot y$ for all $t\in[0,1]$, and
\item $g(0)=g(1)$ is the identity in~$\G$.
\end{Mlist}
Due to $f$ being $\G$-equivariant, there exists a
continuous path $h\c [0,1]\to f^{-1}(y)$ that is defined as $h(t)=g(t)^{-1}\cdot p(t)$.
Because $h(0)\in C_a$ and $h(1)\in C_b$, we find that $C_a=C_b$
and thus arrive at a contradiction.
We conclude that there exists a bijection $\varphi\c I\to J$ \st $C_i=f^{-1}(y)\cap B_{\varphi(i)}$
for all $i\in I$.
\end{proof}

\begin{lemma}
\label{lem:c}
If $G$ is a sparse graph with max-tight decomposition~$\{G_i\}_{i\in I}$,
then the fiber component number of the edge map for $G$ is equal to
\[
c(\cE_G)=c(\cS_G)\cdot \prod_{i\in I} c(\cE_{G_i}).
\]
\end{lemma}

\begin{proof}
Let us consider the diagram at \EQN{dia} and notice that the morphisms $\pi$ and $\mu$ are dominant.
By \RL{em}{a} the edge maps $\cE_G$ and $\cE_{G_i}$ of the sparse graphs $G$ and $G_i$
are dominant for all $i\in I$, and this implies that $\cP_G$ is dominant.
Since the diagram commutes by \Cref{lem:dia}, it follows that the subgraph map $\cS_G$ must be dominant as well.
Hence, the fiber component numbers are well-defined for $\cE_G$, $\cS_G$, $\cP_G$
and $\cE_{G_i}$ for all $i\in I$.
It follows from Sard's theorem that the general fiber is smooth and
thus a component is connected if and only if it is irreducible.
Because the transformation group~$T$ is connected
and the fibers of edge maps associated to tight graphs are $T$-orbits,
it follows that the general fiber of $\cP_G$ has cardinality
\[
c(\cP_G)=\prod_{i\in I}c(\cE_{G_i}).
\]
We claim that $c(\cP_G\circ\cS_G)=c(\cP_G)\cdot c(\cS_G)$.
In general, this is not true for dominant morphisms. For example, if $f\c\C^2\to\C^2$ and $g\c\C^2\to\C$ are defined as $f(x,y):=(x^2,y)$ and $g(x,y):=y$,
then $1=c(g\circ f)\neq c(f)\cdot c(g)=2$. Now, let us set
\[
G=:(V,E),~~~
\{G_i\}_{i\in I}=:\{(V_i,E_i)\}_{i\in I},~~~
X:=\C^{2V}\times T^I,~~~
Y:=\prod_{i\in I}\C^{2V_i},~~~
\G:=T^I.
\]
Notice that $\G$ is a connected topological group with group operation
\[
\bu\circ\bt:=(\bu_i\circ\bt_i)_{i\in I}.
\]
We consider the group actions
\[
\G\times X\to X,~(\bu,(r,\bt))\mapsto (r,\bu\circ\bt)
\quad\text{and}\quad
\G\times Y\to Y,~(\bu,(r^i)_{i\in I})\mapsto(\bu_i\circ r^i)_{i\in I}.
\]
If $E=\varnothing$, then the main assertion holds as a straightforward consequence of the definitions.
In the remainder of the proof, we assume $E\neq\varnothing$
so that $\G$ acts fixed-point free on both $X$ and $Y$.
Since the subgraph map~$\cS_G\c X\to Y$ is $\G$-equivariant
and the fibers of~$\cP_G$ are $\G$-orbits, it follows from
\Cref{lem:fc} that
\[
c(\cP_G\circ\cS_G)=c(\cP_G)\cdot c(\cS_G).
\]
Since the image of a connected component via~$\pi$ is again connected
and $\mu$ is an isomorphism, we find that
\[
c(\cE_G)=c(\mu\circ \cE_G\circ\pi).
\]
We know from \Cref{lem:dia} that the diagram at \EQN{dia} commutes
and thus
\[
c(\mu\circ \cE_G\circ\pi)=c(\cP_G\circ\cS_G).
\]
We conclude that $c(\cE_G)=c(\cP_G)\cdot c(\cS_G)$ as was to be shown.
\end{proof}

The \df{set of embeddings}~$U_V\subset \C^{2V}$
for vertex set~$V$
is defined as the set of realizations~$r\in \C^{2V}$
\st for all sparse graphs $G'=(V',E')$ \st $V'\subset V$,
the Jacobian matrix of the edge map~$\cE_{G'}$
has rank $|E'|$ at $r|_{V'}$.

\begin{remark}
If $M$ is the Jacobian matrix of an edge map~$\cE_G$,
then $\frac{1}{2}\cdot M$ is known as the \df{rigidity matrix}
(see \citep[\textsection1.3.2 and \textsection18.1]{2019hb}).
Elements of the set of embeddings are also known as
\df{regular configurations} (see \citep[\textsection17.2.2]{2019hb}).
\END
\end{remark}

\begin{lemma}[set of embeddings]
\label{lem:U}
Suppose that $V$ is a vertex set.
\begin{claims}
\item\label{lem:U:a}
The set of embeddings~$U_V$ is a non-empty Zariski open subset of $\C^{2V}$.
\item\label{lem:U:b}
If $\Omega=\{G_i\}_{i\in I}$ is a tight decomposition
of a sparse graph~$G=(V,E)$, then
the restricted subgraph map $\cS_\Omega|_{U_V\times T^I}$
is a submersion between the complex manifolds~$U_V\times T^I$ and~$\cS_\Omega(U_V\times T^I)$.
\item\label{lem:U:c}
If $r\in U_V$ is a realization and $u,v,w\in V$ are pairwise different vertices, then
the points $r_u,r_v,r_w\in\C^2$ are not collinear.
\end{claims}

\end{lemma}

\begin{proof}
\ref{lem:U:a}
The set of realizations~$r\in \C^{2V}$
\st the Jacobian matrix of an edge map has lower rank at~$r$
can be expressed in terms of the vanishing of finitely many determinants of submatrices.
Since there are finitely many graphs whose vertex set is contained in~$V$,
we find that $U_V$ is Zariski open.

\ref{lem:U:b}
As a direct consequence of the definitions, we have
\[
\cS_\Omega(U_V\times T^I)\subset \prod_{i\in I} U_{V_i},
\]
where $U_{V_i}\subset\C^{2V_i}$ is a set of embeddings for all $i\in I$.
Let $\fT(\cdot)$ denote the tangent space. It is straightforward to see that if $X,Y$ are complex manifolds, then
\[
\fT(X\times Y)\cong \fT(X)\times\fT(Y).
\]
Thus, by taking the total derivatives of the maps in the diagram at~\EQN{dia},
we obtain the following diagram of linear maps between tangent spaces,
where
\[
g:=\cS_\Omega|_{U_V\times T^I},\quad
f:=\cE_G|_{U_V},\quad
f_i:=\cE_{G_i}|_{U_{V_i}},\quad
F:=\prod_{i\in I}f_i,\quad
p_i:=\bt_i\circ r|_{V_i}:
\]
\[
\begin{tikzcd}
\fT_r(U_V)\times\prod_{i\in I}\fT_{\bt_i}(T)\arrow[d, "d\pi"]\arrow[rr, "dg"] & & \prod_{i\in I}\fT_{p_i}(U_{V_i})\arrow[d,"dF=\prod_{i\in I} df_i"]\\
\fT_r(U_V)\arrow[r,"df"] & \arrow[r,"="]\fT_{f(r)}(\C^E) & \prod_{i\in I}\fT_{f_i(p_i)}(\C^{E_i})
\end{tikzcd}
\]
It follows from \Cref{lem:dia} that the above diagram commutes.
Notice that the total derivatives~$df$ and $dF$ are surjective
by the definitions of sets of embeddings and $G$ being sparse.
We would like to show that $dg$ is surjective as well.
Suppose that $\gamma\in \prod_{i\in I}\fT_{p_i}(U_{V_i})$.
There exists $\alpha\in \fT_r(U_V)$ \st $df(\alpha)=dF(\gamma)$.
Hence, $dg(\alpha,\zv)-\gamma$ is in the kernel of $dF$,
where $\zv$ corresponds to a tuple of~$|I|$ zero vectors.

{\bf Claim 1.}
{\it If $h\c U_{V_i}\times T\to U_{V_i}$ sends $(r,t)$ to $t\circ r$, then for all $i\in I$,}
\[
dh(\{0\}\times \fT_{\bt_i}(T))=\ker df_i.
\]
Notice that $p_i\in U_{V_i}\subset \C^{2V_i}$ is a realization.
An element in the tangent space~$\fT_{p_i}(U_{V_i})$ assigns to each of the points in~$P:=\{(p_i)_v\}_{v\in V_i}\subset \C^2$
a vector in~$\C^2$.
If an element in~$\fT_{p_i}(U_{V_i})$ assigns to each point in~$P$ the velocity vector of some transformation~$t\in T$,
then it must be in the kernel of~$df_i$ (see \citep[\textsection18.1]{2019hb}).
Since $\dim T=3$, we find that $\dim \ker df_i\geq 3$.
The kernel of $df_i$ corresponds to the kernel of Jacobian matrix of $\cE_{G_i}$ at~$p_i$
and since $p_i$ is contained in the set of embeddings~$U_{V_i}$,
this Jacobian matrix has rank $2|V_i|-3$.
It follows that each element in~$\ker df_i\subset \fT_{p_i}(U_{V_i})$ comes from
transformation in $T$,
which is equivalent to the statement of Claim~1.
For more details we refer to \citep[\textsection1.3.2 and Theorem~18.4]{2019hb},
where it is explained that the kernel of the ``rigidity matrix'' (corresponding to $\ker df_i$)
of an ``infinitesimally rigid framework'' (corresponding to $(G,p_i)$),
does not contain ``infinitesimal motions''.

It follows from Claim~1 that
\[
dg\Bigl(\{0\}\times\prod_{i\in I}\fT_{\bt_i}(T)\Bigr)=\ker dF.
\]
Hence, there exists $\beta\in\prod_{i\in I}\fT_{\bt_i}(T)$ \st $dg(0,\beta)=dg(\alpha,\zv)-\gamma$,
which implies that $\gamma=dg(\alpha,\beta)$.
We conclude that $g$ is a submersion as was to be shown.

\ref{lem:U:c}
If $G'$ is a triangle with vertices $V'=\{u,v,w\}\subset V$
and $s\in \C^{2V}$ a realization \st $s_u,s_v,s_w\in\C^2$ are collinear,
then the Jacobian matrix of the edge map~$\cE_{G'}$
has rank strictly less than $2|V'|-3$ at $s|_{V'}\in \C^{2V'}$ and thus $s\notin U_V$
(see \citep[18.1]{2019hb}).
\end{proof}

\begin{remark}
Alternatively, we could define the set of embeddings $U_V$ as
the set of realizations in $\C^{2V}$ that satisfy both \RLS{U}{b}{U}{c}.
A similar proof as for \RL{U}{a} would show that $U_V \subset \C^{2V}$
is Zariski open.
\END
\end{remark}

We call an algebraic variety \df{irreducible} if it is not the union
of two Zariski closed sets.
We remark that a Zariski open set of a variety is again a variety.

\begin{lemma}[fiber product]
\label{lem:XxX}
Suppose that $f\c X\to Y$ is a dominant morphism between
smooth irreducible varieties and let
\[
X\times_f X:=\set{(p,q)\in X^2}{f(p)=f(q)}.
\]
\begin{claims}
\item\label{lem:XxX:a}
If $f$ is a submersion, then the fiber product $X\times_f X$ is a complex manifold and
\[
\dim X\times_f X=2\dim X-\dim Y.
\]
\item\label{lem:XxX:b}
If $X\times_f X$ is irreducible, then general fiber of $f$ is irreducible.
\end{claims}
\end{lemma}

\begin{proof}
See \APP{fiber}.
\end{proof}

If $G=(V,E)$ is a graph
with tight decomposition~$\Omega:=\{(V_i,E_i)\}_{i\in I}$,
then its \df{fiber square}
is defined as
\[
P(\Omega):=\set{(r,\bt;s,\bu)\in \left(U_V\times T^I\right)^2}{\cS_\Omega(r,\bt)=\cS_\Omega(s,\bu)}.
\]
If $G$ is sparse and $\Omega$ a max-tight decomposition, then we write $P(G)$ instead of~$P(\Omega)$.

We remark that our definition of ``fiber square'' deviates from the literature, where this notion
instead refers to the diagram associated to a fiber product.

\begin{lemma}
\label{lem:cS}
If $G$ is a sparse graph \st its fiber square~$P(G)$ is irreducible, then
the fiber component number of the edge map $\cS_G$ is equal to
$c(\cS_G)=1$.
\end{lemma}

\begin{proof}
Recall from \RL{U}{a} that $U_V\subset\C^{2V}$ is a Zariski open set.
It now follows from \RL{XxX}{b} that the general fiber of $\cS_G|_{U_V\times T^I}$
is irreducible. Hence, the general fiber
of $\cS_G$ is irreducible as well.
\end{proof}

\begin{proposition}
\label{prp:main}
If $G=(V,E)$ is a sparse graph
with max-tight decomposition $\{G_i\}_{i\in I}$
and $P(G)$ is irreducible,
then for general edge length assignment $\lambda\in\C^E$,
\[
c(G,\lambda)=\prod_{i\in I} c(G_i).
\]
\end{proposition}

\begin{proof}
We know from \RL{em}{b} and \Cref{lem:c} that
\[
c(G,\lambda)=c(\cE_G)=c(\cS_G)\cdot \prod_{i\in I} c(\cE_{G_i})=c(\cS_G)\cdot\prod_{i\in I} c(G_i).
\]
If $P(G)$ is irreducible, then $c(\cS_G)=1$ by \Cref{lem:cS}.
\end{proof}

\section{Associated fiber squares are irreducible}
\label{sec:red}

In this section, we prove by induction that the
fiber square~$P(G)$ associated to a sparse graph~$G$ is irreducible.
For the induction step,
we construct a new sparse graph~$G^\m$ that has one vertex less than $G$
and show that if $P(G^\m)$ is irreducible, then $P(G)$ must be irreducible as well.
To prove the induction step, we construct a ``comparison map'' $P(G)\to P(G^\m)$
and analyze the irreducibility and dimension of each of its fibers.
For the induction basis, we observe that $P(G)$ is irreducible if $G$ has one vertex.
We conclude this section with a proof for \Cref{thm:main,cor:main}
by applying \Cref{prp:main} from the previous section.

Suppose that $G=(V,E)$ is a sparse graph
and let $G'$ be the graph that is obtained from~$G$ by removing the vertex $v\in V$.
For an initial understanding,
let us assume that each max-tight subgraph of the sparse graph~$G'$ is
either
\begin{Mlist}
\item a max-tight subgraph of~$G$, or a max-tight subgraph of $G$ minus the vertex~$v$.
\end{Mlist}
In this case, we obtain a canonical map from
the fiber square of~$G$ to the fiber square of~$G'$ by restricting
the realization components to the vertex set~$V':=V\setminus\{v\}$.
In other words, the comparison map sends $(r,\bt;s,\bu)$ to $(r|_{V'},\bt;s|_{V'},\bu)$.
In the formal definition of the comparison map below,
we will also include the case that the above assumption does not hold.
In the more general setting, we may need to add an additional edge to~$G'$.

Let $G=(V,E)$ and $G'=(V',E')$ be graphs.
We denote by $E[v]$ the set of edges in $E$ that contain the vertex~$v\in V$.
The degree of a vertex~$v\in V$ is defined as the cardinality of~$E[v]$.
If $v\notin V$, then $E[v]:=\varnothing$.

We call $G$ a \df{0-extension} of $G'$ with respect to vertex~$v_0$ if
\[
(V,E)=\bigl(V'\cup\{v_0\},E'\cup\set{\{v_0,v_i\}}{i\in\{1,2\}}),
\]
where $v_0\notin V'$ and $v_1,v_2\in V'$ are distinct vertices (see \Cref{fig:H2} left).

We call $G$ a \df{1-extension} of $G'$ with respect to vertex~$v_0$ and edge~$\{v_2,v_3\}$
if
\[
(V,E)=\bigl(V'\cup\{v_0\},E'\cup\set{\{v_0,v_i\}}{i\in\{1,2,3\}}\setminus\{\{v_2,v_3\}\}\bigr),
\]
where $v_0\notin V'$, $v_1,v_2,v_3\in V'$ are pairwise distinct vertices and $\{v_2,v_3\}\in E'$ is an edge (see \Cref{fig:H2} right).

We remark that 0- and 1-extensions are also known as ``Henneberg moves'' \cite{2019hb}.
\begin{figure}[!ht]
\centering
\begin{tabular}{cc@{\hspace{15mm}}cc}
\begin{tikzpicture}[scale=0.7]
\draw [dashed] (0,0) ellipse (2cm and 1cm);
\node[vertex] (v) at (-1,0) [label=below:$v_1$] {};
\node[vertex] (w) at (1,0) [label=below:$v_2$] {};
\node at (0,-2) {$G'$};
\end{tikzpicture}
&
\begin{tikzpicture}[scale=0.7]
\draw [dashed] (0,0) ellipse (2cm and 1cm);
\node[vertex] (v) at (-1,0) [label=below:$v_1$] {};
\node[vertex] (w) at (1,0) [label=below:$v_2$] {};
\node[vertex] (k) at (0,2) [label=above:$v_0$] {};
\node at (0,-2) {$G$};
\path (v) edge[edge] (k) (w) edge[edge] (k);
\end{tikzpicture}
&
\begin{tikzpicture}[scale=0.7]
\draw [dashed] (0,0) ellipse (2cm and 1.5cm);
\node[vertex] (v) at (0,1) [label=left:$v_3$] {};
\node[vertex] at (-1,-0.1) [label=below:$v_1$] {};
\node[vertex] (w) at (1,-0.1) [label=below:$v_2$] {};
\node at (0,-2) {$G'$};
\path (v) edge[edge] (w);
\end{tikzpicture}
&
\begin{tikzpicture}[scale=0.7]
\draw [dashed] (0,0) ellipse (2cm and 1.5cm);
\node[vertex] (i) at (-1,-0.1) [label=below:$v_1$] {};
\node[vertex] (j) at (1,-0.1) [label=below:$v_2$] {};
\node[vertex] (k) at (0,1) [label=left:$v_3$] {};
\node[vertex] (w) at (2,2) [label=above:$v_0$] {};
\path (i) edge[edge] (w)
      (j) edge[edge] (w)
      (k) edge[edge] (w);
\node at (0,-2) {$G$};
\end{tikzpicture}
\end{tabular}
\caption{The left/right graph $G$ is a 0/1-extension of the graph~$G'$.}
\label{fig:H2}
\end{figure}
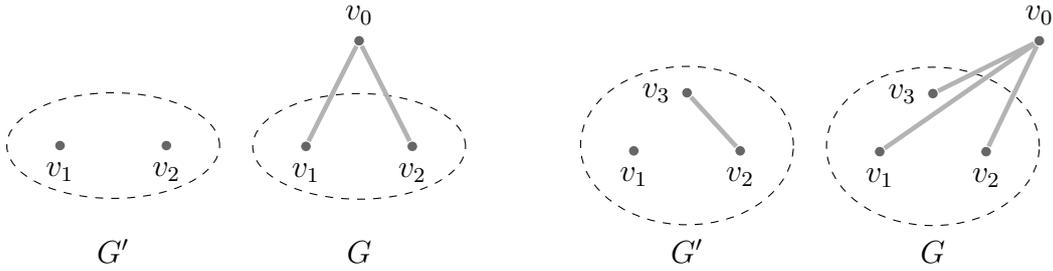

Suppose that $G=(V,E)$ be a sparse graph with max-tight decomposition
\[
\{G_i\}_{i\in I}=\{(V_i,E_i)\}_{i\in I}.
\]
A \df{marking} for $G$ is a defined as a function $\m\in V^Z$ \st
\begin{Mlist}
\item
$\m\c Z\to V$ is injective, $Z\in\{\{0\},\{0,1\},\{0,1,2\},\{0,1,2,3\}\}$
and $|V|>1$,
\item
$E[\m_0]=\set{\{\m_0,\m_z\}}{z\in Z\setminus\{0\}}$, and
\item
if $|Z|=4$ and $\m(Z)\subset V_i$ for some $i\in I$,
then $G_i$ is a 1-extension of a tight graph \wrt $\m_0$ and edge~$\{\m_2,\m_3\}$.
\end{Mlist}
Suppose that $\m\in V^Z$ is a marking for $G$.
For all $i\in I$, we define
\begin{align*}
V^\m_i&:=V_i\setminus\m_0,
\\
E^\m_i&:=
\begin{cases}
E_i\setminus E_i[\m_0]\cup\{\{\m_2,\m_3\}\} & \text{if $|Z|=4$ and $\m(Z)\subset V_i$ for some $i\in I$, or}\\
E_i\setminus E_i[\m_0] & \text{otherwise}.
\end{cases}
\end{align*}
We set
\[
I^\m:=\set{i\in I}{V_i\neq\{\m_0\} \text{ and } |E_i[\m_0]|\neq1}.
\]
For all $i\in I^\m$, we define
\[
G^\m_i:=(V^\m_i,E^\m_i).
\]
The \df{reduced graph} associated to the marking~$\m$ is defined as
\[
G^\m=(V^\m,E^\m):=\bigcup_{i\in I^\m} G^\m_i.
\]
If $G^\m$ is sparse with max-tight decomposition~$\{G^\m_i\}_{i\in I}$, then
the \df{comparison map} associated to the marking~$\m$
is (if it exists) defined as
\[
\CM\c P(G)\to P(G^\m),\quad
(r,\bt;s,\bu)
\mapsto
(r|_{V^\m},\bt|_{I^\m};s|_{V^\m},\bu|_{I^\m}).
\]

\begin{lemma}
\label{lem:cm}
If $G=(V,E)$ is a sparse graph \st $|V|>1$, then
there exists a marking~$\m\in V^Z$ \st
\begin{Mlist}
\item its reduced graph~$G^\m$ is sparse with max-tight decomposition~$\{G^\m_i\}_{i\in I^\m}$, and
\item its comparison map $\CM\c P(G)\to P(G^\m)$ exists.
\end{Mlist}
\end{lemma}

\begin{proof}
A 0- or 1-extension of a tight graph is again tight.
Moreover, any tight graph can be obtained from an edge by applying subsequently 0- and/or 1-extensions
(see \cite{1911,1970} or alternatively \citep[Theorem~19.2]{2019hb}).
Hence, for all tight graphs there exists a marking.
The sparse graph~$G$ is a subgraph of a tight graph
and a marking~$\m$ for this tight graph is also a marking for~$G$ as a direct consequence of the definitions.

Suppose that $i\in I^\m$ so that $V_i\neq\{\m_0\}$ and $|E_i[\m_0]|\neq 1$ by definition of~$I^\m$.
We make a case distinction on $|E_i[\m_0]|$.

First, suppose that $|E_i[\m_0]|=0$. In this case, $\m_0\notin V_i$,
so that $G^\m_i=G_i$ is a max-tight subgraph of $G^\m$.

Next, suppose that $|E_i[\m_0]|=2$.
In this case $G_i$ is a 0-extension of $G^\m_i$.
Since $|E_i|-2=2(|V_i|-1)-3$, we have $|E^\m_i|=2|V^\m_i|-3$ so that $G^\m_i$ is tight.
It follows that $G^\m_i$ is a max-tight subgraph of $G^\m$.

Finally, suppose that $|E_i[\m_0]|=3$.
By the definition of markings,
we find that $G_i$ is a 1-extension of the tight graph $G_i^\m$ \wrt $\m_0$ and the edge~$\{\m_2,\m_3\}$.
It follows that $G^\m_i$ is a max-tight subgraph of $G^\m$.

We have established that $\{G^\m_i\}_{i\in I^\m}$ is a max-tight decomposition for $G^\m$
so that $G^\m$ is sparse and the fiber square $P(G^\m)$ is well-defined.
By the definition of the set of embeddings~$U_V$ and since $V^\m\subset V$,
we find that $\set{r|_{V^\m}}{r\in U_V}\subset U_{V^\m}$.
Hence, the comparison map~$\CM$ is also well-defined.
\end{proof}

Suppose that $G=(V,E)$ is a sparse graph with marking~$\m\in V^Z$.
In order to show that $P(G)$ is irreducible,
we want to show that the hypothesis of \Cref{lem:if} below
holds if $f$ is equal to the comparison map $\CM$.

\begin{lemma}
\label{lem:if}
Suppose that $f\c X\to Y$ is a morphism \st $X$ is equidimensional and $Y$ is irreducible.
If for general $y\in Y$,
\begin{Mlist}
\item
$\dim f^{-1}(y)=\dim X-\dim Y$,
\item
$f^{-1}(y)$ is irreducible, and
\item
$
\dim\set{z\in Y}{\dim f^{-1}(z)=\dim X-\dim Y+c}
<
\dim Y-c
$
for all $c\in\Z_{>0}$,
\end{Mlist}
then $X$ is irreducible.
\end{lemma}

\begin{proof}
It follows from the theorem on fiber dimension
\citep[\textsection I.6.3, Theorem~7]{1994} applied to the surjective morphism~$f\c X\to f(X)$
and the assumption on the dimension of the general fiber of $f$
that $\dim f(X)=\dim Y$ and thus $f$ is dominant.
Since $X$ has finitely many irreducible components, there
exists a component~$X_0\subset X$ \st $\dim f(X_0)=\dim Y$.
Now suppose by contradiction that there exists an irreducible component $X_1\subset X$
different from $X_0$.
By assumption, the fiber~$f^{-1}(y)$ is irreducible for general $y\in f(X_0)$
and thus $f^{-1}(y)\subset X_0$.
Hence, $\dim f(X_1)=\dim Y-c$ for some $c>0$.
Notice that $\dim X=\dim X_0=\dim X_1$ as $X$ is equidimensional.
We again apply the fiber dimension theorem
\citep[\textsection I.6.3, Theorem~7]{1994}
to the restricted map~$f|_{X_1}$,
and find that for general~$z\in f(X_1)$,
\[
\dim f^{-1}(z)=\dim X-(\dim Y-c).
\]
We arrived at a contradiction as $\dim f(X_1)=\dim Y-c$ and
\[
\dim f(X_1)\leq \dim\set{z\in Y}{\dim f^{-1}(z)=\dim X-\dim Y+c}<\dim Y-c.
\]
We established that $X=X_0$ and thus $X$ is irreducible.
\end{proof}

In order to analyze the irreducibility and dimension of the fibers
of the comparison map, we first address in \Cref{lem:fib} below the following sub-problem.
Suppose that
$W\subset\{1,2,3\}$ and $p,q\in(\C^2)^W$ is a pair of realizations.
We would like to determine the irreducibility and dimension of the set of pairs of realizations
$r,s\in (\C^2)^{W\cup\{0\}}$
\st $r|_W=p$, $s|_W=q$ and
$d(r_0,r_w)=d(s_0,s_w)$ for all $w\in W$.
If $1,2\in W$, then we are also interested in the case that
additionally the triangles spanned by $\{r_0,r_1,r_2\}$ and $\{s_0,s_1,s_2\}$
are related by some transformation in~$T$ (and thus are not reflections of each other).

\begin{lemma}
\label{lem:fib}
Suppose that $W\subset\{1,2,3\}$ and let $P:=\{p_w\}_{w\in W}$ and $Q:=\{q_w\}_{w\in W}$
be finite sets of points in~$\C^2$ \st $|P|=|W|$ and if $|W|=3$, then $P$ is not collinear.
Let
\begin{align*}
\Lambda     &:=\set{\{v,w\}\subset W}{\sd(p_v,p_w)=\sd(q_v,q_w) \text{ and }v\neq w},\\
\Gamma&:=\set{(p_0,q_0)\in\C^2\times\C^2}{\sd(p_0,p_w)=\sd(q_0,q_w) \text{ for all } w\in W},\\
\Gamma'&:=\set{(p_0,q_0)\in \Gamma}{t(p_0)=q_0,~ t(p_1)=q_1 \text{ and } t(p_2)=q_2 \text{ for some }t\in T}.
\end{align*}
\begin{claims}
\item\label{lem:fib:a}
If $W=\varnothing$,
then $\Gamma$ is irreducible and $\dim \Gamma=4$.
\item\label{lem:fib:b}
If $W=\{1\}$,
then $\Gamma$ is irreducible and $\dim \Gamma=3$.
\item\label{lem:fib:c}
If $W=\{1,2\}$ and $\Lambda=\{\{1,2\}\}$,
then $\Gamma'$ is irreducible and
\[
\dim \Gamma=\dim\Gamma'=2.
\]
\item\label{lem:fib:d}
If $W=\{1,2\}$ and $\Lambda=\varnothing$, then $\Gamma$ is irreducible and $\dim \Gamma=2$.
\item\label{lem:fib:e}
If $W=\{1,2,3\}$ and $\Lambda=\{\{1,2\},\{1,3\},\{2,3\}\}$,
then $\Gamma'$ is irreducible and
\[
\dim \Gamma=\dim\Gamma'=2.
\]
\item\label{lem:fib:f}
If $W=\{1,2,3\}$ and $\Lambda=\{\{1,2\},\{2,3\}\}$,
then $\Gamma'$ is irreducible and
\[
\dim \Gamma=\dim\Gamma'=1.
\]
\item\label{lem:fib:g}
If $W=\{1,2,3\}$ and $\Lambda=\{\{1,2\}\}$,
then $\Gamma'$ is irreducible and
\[
\dim \Gamma=\dim\Gamma'=1.
\]
\item\label{lem:fib:h}
If $W=\{1,2,3\}$ and $\Lambda=\varnothing$,
then $\Gamma$ is irreducible and $\dim \Gamma=1$.
\end{claims}
\end{lemma}

\begin{proof}
\ref{lem:fib:a} We have $\Gamma=\C^2\times\C^2$ as a direct consequence of the definitions.

If $1\in W$, then we define
\[
D:=\set{t\in T}{p_1=t(q_1)}.
\]
We observe that $D$ is irreducible and of dimension~$1$ as its elements
are the composition of a complex translation with a complex
rotation centered around $p_1$.

\ref{lem:fib:b} The map $D\times\C^2\dto \Gamma$ that sends $(t,p_0)$ to $(p_0,t^{-1}(p_0))$
is birational.
Indeed, we can recover $(t,p_0)$ uniquely from $(p_0,q_0)\in \Gamma$ and $(p_1,q_1)$
if $\{p_0,q_0,p_1\}$ are pairwise distinct since by definition $p_0=t(q_0)$ and $p_1=t(q_1)$.
Hence, $\Gamma$ is irreducible and $\dim \Gamma=\dim D\times\C^2=3$.

For the remainder of the proof, we assume that $1,2\in W$.
Let
$t_*\c\C^2\to \C^2$ denote the reflection along the line spanned by $p_1,p_2\in\C$
and
\[
\Gamma_*:=\set{(p_0,q_0)\in \Gamma}{(t_*\circ t)(p_0)=q_0,~ t(p_1)=q_1 \text{ and } t(p_2)=q_2 \text{ for some }t\in T}.
\]
We observe that $\Gamma=\Gamma'\cup\Gamma_*$ and thus $\dim\Gamma=\dim \Gamma'$.

\ref{lem:fib:c},\ref{lem:fib:e}
As $P$ is not collinear there exists a unique transformation~$t\in T$ \st
$p_w=t(q_w)$ for all $w\in W$.
Hence, the map $\C^2\dto \Gamma'$ that sends $p_0$ to $(p_0,t^{-1}(p_0))$
is birational, which implies that $\Gamma'$ is irreducible and $\dim\Gamma'=2$.

The \df{bisector} of different $a,b\in\C^2$ is defined as
\[
L_{a,b}:=\set{m\in\C^2}{\sd(m,a)=\sd(m,b)}.
\]
\ref{lem:fib:d}
For all $t\in D$, we have $p_2\neq t(q_2)$ since $p_1=t(q_1)$
and $\{1,2\}\notin \Lambda$.
Hence, there exists a birational map $\ell_t\c \C\dto L_{p_2,t(q_2)}$
that depends algebraically on the parameter~$t\in D$.
Now suppose that $(p_0,q_0)\in\Gamma$.
There exists a unique $t\in D$ \st $p_0=t(q_0)$.
Moreover, there exists a unique $x\in \C$ \st $p_0=\ell_t(x)$ since
the map~$\ell_t$ is birational and
\[
\sd(p_0,p_2)=\sd(q_0,q_2)=\sd(t^{-1}(p_0),q_2)=\sd(p_0,t(q_2)).
\]
Hence, the following map is birational:
\[
D\times\C \dto \Gamma, \quad (t,x)\mapsto \bigl(\ell_t(x),~(t^{-1}\circ\ell_t)(x)\bigr).
\]
This implies that $\Gamma$ is irreducible and $\dim\Gamma=\dim D\times\C=2$.

\ref{lem:fib:f},\ref{lem:fib:g}
As $P$ is not collinear, we have $p_1\neq p_2$.
Since $\{1,2\}\in\Lambda$,
there exists a unique~$t\in D$ \st $p_2=t(q_2)$.
Since $p_1=t(q_1)$ and $\{1,3\}\notin \Lambda$,
we have $p_3\neq t(q_3)$ and thus $L_{p_3,t(q_3)}$ is a line.
Hence, there exists a birational map
\[
\ell\c \C\dto L_{p_3,t(q_3)}.
\]
Now suppose that $(p_0,q_0)\in\Gamma'$ and notice that $p_0=t(q_0)$.
Since $\ell$ is birational and
\[
\sd(p_0,p_3)=\sd(q_0,q_3)=\sd(t^{-1}(p_0),q_3)=\sd(p_0,t(q_3)),
\]
there exists a unique $x\in \C$ \st $p_0=\ell(x)$.
We established that the following map is birational:
\[
\C \dto \Gamma', \quad x\mapsto \bigl(\ell(x),~(t^{-1}\circ\ell)(x)\bigr).
\]
This implies that $\Gamma'$ is irreducible and
$\dim\Gamma'=\dim\Gamma=1$.

\ref{lem:fib:h}
For all $t\in D$, we have $p_2\neq t(q_2)$ and $p_3\neq t(q_3)$
since $\{1,2\},\{1,3\}\notin \Lambda$ and $p_1=t(q_1)$.
Moreover, $\sd(p_2,p_3)\neq\sd(t(q_2),t(q_3))=\sd(q_2,q_3)$ since $\{2,3\}\notin\Lambda$.
Therefore, the bisectors $L_{p_2,t(q_2)}$ and $L_{p_3,t(q_3)}$ do not coincide so that
\[
|L_{p_2,t(q_2)}\cap L_{p_3,t(q_3)}|\leq 1.
\]
There exist at most finitely many $t\in D$ such that the
lines $L_{p_2,t(q_2)}$ and $L_{p_3,t(q_3)}$
are parallel and thus the map $\ell\c D\dto \ell(D)\subset\C^2$
that sends $t$ to the unique intersection point in~$L_{p_2,t(q_2)}\cap L_{p_3,t(q_3)}$
is birational.
Now suppose that $(p_0,q_0)\in\Gamma$ is general.
There exists a unique $t\in D$ \st $p_0=t(q_0)$.
For all $w\in\{2,3\}$, we require that
\[
\sd(p_0,p_w)=\sd(q_0,q_w)=\sd(t^{-1}(p_0),q_w)=\sd(p_0,t(q_w)).
\]
This implies that $p_0\in L_{p_2,t(q_2)}\cap L_{p_3,t(q_3)}$ and thus $p_0=\ell(t)$.
We established that the following map is birational:
\[
D\dto \Gamma,\quad t\mapsto (\ell(t),(t^{-1}\circ\ell)(t)).
\]
We conclude that $\Gamma$ is irreducible and $\dim\Gamma=1$.
\end{proof}

\begin{lemma}
\label{lem:dim}
If $G=(V,E)$ is a graph with
tight decomposition
\[
\Omega:=\{(V_i,E_i)\}_{i\in I},
\]
then its fiber square~$P(\Omega)$ is a complex manifold of dimension
\[
\dim P(\Omega)=4|V|+6|I|-2\sum_{i\in I}|V_i|.
\]
\end{lemma}

\begin{proof}
It follows from \Cref{lem:U,lem:XxX} that
$P(G)$ is complex manifold of dimension
$2\dim(U_V\times T^I)-\dim\prod_{i\in I}\C^{2V_i}$.
Since $\dim T=3$, we find that $\dim P(G)$ is as asserted.
\end{proof}

Let $G=(V,E)$ is a sparse graph with marking $\m\in V^Z$
\st $(\m_0,\m_1,\m_2,\m_3)=(0,1,2,3)$
and suppose that a tight subgraph of~$G$ that has $2$ as one of its vertices
contains neither the vertex~$1$ nor the vertex~$3$.
In this case, there may exist points $(r,\bt;s,\bu)$ in the fiber square~$P(G^\m)$
of the reduced graph~$G^\m$
whose fibers \wrt the comparison map $\CM\c P(G)\to P(G^\m)$
has higher dimension than expected.
Indeed, if
$d(r_1,r_3)=d(s_1,s_3)$ and $d(r_2,r_3)=d(s_2,s_3)$, then \RL{fib}{e}
suggests that the dimension is one instead of two.
Thus, in order to satisfy the hypothesis of \Cref{lem:if},
we need to show that the following subset has codimension two in $P(G^\m)$:
\[
\Psi:=\set{(r,\bt;s,\bu)\in P(G^\m)}{\sd(r_1,r_3)=\sd(s_1,s_3) \text{ and } \sd(r_2,r_3)=\sd(s_2,s_3)}.
\]
For this purpose, we construct the graph~$G'$ by
adding to $G^\m$ the edges~$\{1,3\}$ and $\{2,3\}$,
and show that $G'$ is sparse.
Next, we apply the dimension formula for fiber squares at \Cref{lem:dim}
to both $P(G^\m)$ and $P(\Omega')$,
where $\Omega'$ is a certain tight decomposition of $G'$.
As $P(\Omega')$ can be projected onto $\Psi\subset P(G^\m)$, we show that $\codim\Psi=2$.
The tight-decomposition $\Omega'$ is constructed by appending to the unique max-tight decomposition of~$G$
the two single-edge graphs consisting of the edge $\{1,3\}$ or $\{2,3\}$.
Notice that the max-tight decomposition of $G'$ may be trivial (see for example \Cref{fig:codim}).
For this reason, we consider fiber squares \wrt tight-decompositions that are not necessarily max-tight.
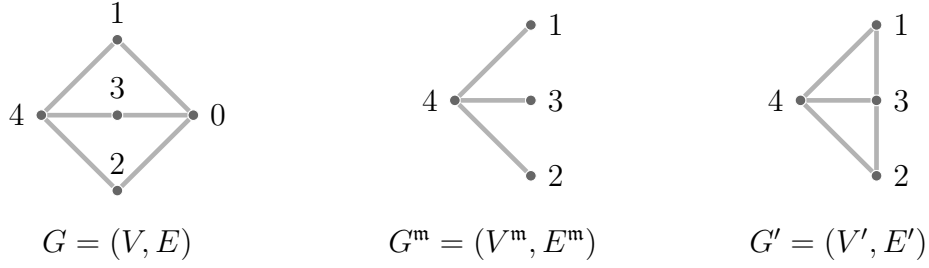
\begin{figure}[!ht]
\centering
\csep{1cm}
\begin{tabular}{ccc}
\begin{tikzpicture}
\node[vertex] (v) at (-1,0) [label=left:$4$] {};
\node[vertex] (i) at (0,1) [label=above:$1$] {};
\node[vertex] (j) at (0,0) [label=above:$3$] {};
\node[vertex] (k) at (0,-1) [label=above:$2$] {};
\node[vertex] (w) at (1,0) [label=right:$0$] {};
\path (v) edge[edge] (i)
      (v) edge[edge] (j)
      (v) edge[edge] (k)
      (w) edge[edge] (i)
      (w) edge[edge] (j)
      (w) edge[edge] (k);
\end{tikzpicture}
&
\begin{tikzpicture}
\node[vertex] (v) at (-1,0) [label=left:$4$] {};
\node[vertex] (i) at (0,1) [label=right:$1$] {};
\node[vertex] (j) at (0,0) [label=right:$3$] {};
\node[vertex] (k) at (0,-1) [label=right:$2$] {};
\path (v) edge[edge] (i)
      (v) edge[edge] (j)
      (v) edge[edge] (k);
\end{tikzpicture}
&
\begin{tikzpicture}
\node[vertex] (v) at (-1,0) [label=left:$4$] {};
\node[vertex] (i) at (0,1) [label=right:$1$] {};
\node[vertex] (j) at (0,0) [label=right:$3$] {};
\node[vertex] (k) at (0,-1) [label=right:$2$] {};
\path (v) edge[edge] (i)
      (v) edge[edge] (j)
      (v) edge[edge] (k)
      (j) edge[edge] (i)
      (j) edge[edge] (k);
\end{tikzpicture}
\\
$G=(V,E)$ & $G^\m=(V^\m,E^\m)$ & $G'=(V',E')$
\end{tabular}
\caption{Three sparse graphs. The max-tight subgraphs
of $G$ and $G^\m$ consist of single edges. The graph $G'$ is tight and thus its max-tight decomposition is trivial.}
\label{fig:codim}
\end{figure}
%

\begin{lemma}
\label{lem:codim}
If $G=(V,E)$ be a sparse graph with marking $\m\in V^Z$
\st $(\m_0,\m_1,\m_2,\m_3)=(0,1,2,3)$
and
$\{1,3\},\{2,3\}\nsubseteq V^\m_i$ for all $i\in I^\m$,
then the following subset has codimension two in the fiber square~$P(G^\m)$:
\[
\Psi:=\set{(r,\bt;s,\bu)\in P(G^\m)}{\sd(r_1,r_3)=\sd(s_1,s_3) \text{ and } \sd(r_2,r_3)=\sd(s_2,s_3)},
\]
where $\{(V^\m_i,E^\m_i)\}_{i\in I^\m}$ is the max-tight decomposition of the reduced graph~$G^\m$.
\end{lemma}

\begin{proof}
Since $\{1,3\}\nsubseteq V^\m_i$ for all~$i\in I^\m$,
we find that the graph $(V^\m,E^\m\cup\{\{1,3\}\})$ is sparse.
Suppose by contradiction that the graph~$G':=(V^\m,E^\m\cup\{\{1,3\},\{2,3\}\}$ is not sparse.
Then there must exists $j\in I^\m$ \st
$(V^\m_j,E^\m_j\cup\{\{1,3\},\{2,3\}\})$ is not sparse and becomes tight
after we remove the edge $\{2,3\}$.
We arrived at a contradiction as $\{1,2,3\}\subset V^\m_j$
and thus we established that $G'$ is sparse.

Let $H_{u,v}$ denote the graph consisting of a single edge~$\{u,v\}$.
We set
$I':=I^\m\cup\{a,b\}$ for some indices~$a,b\notin I^\m$.
Let $G'_a:=H_{1,3}$,
$G'_b:=H_{2,3}$
and
$G'_i:=G^\m_i$
for all $i\in I^\m$.
Notice that the sparse graph~$G'$ has the following tight decomposition:
\[
\Omega':=\{G'_i\}_{i\in I'}=\{(V'_i,E'_i)\}_{i\in I'}.
\]
We find that $|V'|-|V^\m|=0$,
$|I'|-|I^\m|=2$ and
\[
\sum_{i\in I'}|V'_i|-\sum_{i\in I^\m}|V^\m_i|
=
|V'_a|+|V'_b|=4.
\]
Thus, the dimension formula at \Cref{lem:dim} shows that
\[
\dim P(\Omega')-\dim P(G^\m)=4\left(|V'|-|V^\m|\right)+6\left(|I'|-|I^\m|\right)-2\left(|V'_a|+|V'_b|\right)=4.
\]
Since $I^\m\subset I'$ by construction, there exists a projection
\[
\tau\c P(\Omega')\to P(G^\m),\quad
(r,\bt;s,\bu)\mapsto(r,\bt|_{I^\m};s,\bu|_{I^\m}).
\]
We would like to determine the dimension of the fiber of the projection~$\tau$
at a point $(r,\bt|_{I^\m};s,\bu|_{I^\m})$.
The transformation tuple~$(\bt_a,\bt_b)\in T^2$ can be arbitrary.
The transformation~$\bt_a$ uniquely determines~$\bu_a$,
since $\bt_a(r_i)=\bu_a(s_i)$ for $i\in\{1,3\}$.
Similarly, $\bt_b$ uniquely determines~$\bu_b$.
We deduce that all the fibers of~$\tau$ have dimension
$2\dim T=6$ and thus
\[
\dim \tau(P(\Omega'))=\dim P(\Omega')-6=\dim P(G^\m)-2.
\]
We find that $\Psi=\tau(P(\Omega'))$ as a straightforward consequence of the definitions
and thus we established that $\Psi\subset P(G^\m)$ has codimension two.
\end{proof}

In the following \Cref{lem:ifc}, we characterize the dimensions and irreducibility of
the fibers of a comparison map $\CM\c P(G)\to P(G^\m)$.
This shows that the hypothesis of \Cref{lem:if} is satisfied
and thus we can conclude that if $P(G^\m)$ is irreducible,
then $P(G)$ must be irreducible as well.

\begin{lemma}
\label{lem:ifc}
Suppose that $G=(V,E)$ is a sparse graph with marking $\m\in V^Z$
and let $F_y:=\CM^{-1}(y)$ for $y\in P(G^\m)$
denote the fiber of the comparison map $\CM$.
Then, for general $p\in P(G^\m)$ and all $c\in\Z_{>0}$ the following holds:
\begin{Mlist}
\item $F_p$ is irreducible,
\item $\dim F_p=\dim P(G)-\dim P(G^\m)$, and
\item $\dim \set{y\in P(G^\m)}{\dim F_y=\dim P(G)-\dim P(G^\m)+c}<\dim P(G^\m)-c$.
\end{Mlist}
\end{lemma}

\begin{proof}
Let $\{(V_i,E_i)\}_{i\in I}$
and
$\{(V^\m_i,E^\m_i)\}_{i\in I^\m}$
be the max-tight decompositions of the graphs $G=(V,E)$ and $G^\m=(V^\m,E^\m)$, \resp.
For ease of notation, we assume \Wlog that
$\m_z=z$ for all $z\in Z$.
We consider the following arbitrary but fixed point in $P(G^\m)$:
\[
y:=(\tr,\tbt;\ts,\tbu).
\]
In accordance with \Cref{lem:fib}, we define the following sets
that depend on $\tr,\ts$:
\[
W:=Z\setminus\{0\},\qquad
P:=\{\tr_i\}_{i\in W},\qquad
Q:=\{\ts_i\}_{i\in W},
\]
\begin{align*}
\Lambda &:=\set{\{v,w\}\subset W}{\sd(\tr_v,\tr_w)=\sd(\ts_v,\ts_w) \text{ and }v\neq w},\\
\Gamma  &:=\set{(r_0,s_0)\in\C^2\times\C^2}{\sd(r_0,\tr_w)=\sd(s_0,\ts_w) \text{ for all } w\in W},\\
\Gamma' &:=\set{(r_0,s_0)\in \Gamma}{t(r_0)=s_0,~ t(\tr_1)=\ts_1 \text{ and } t(\tr_2)=\ts_2 \text{ for some }t\in T}.
\end{align*}
Notice that $\tr,\ts\in U_{V^\m}$ by definition so that
$|P|=|W|$ and if $|W|=3$, then $P,Q\subset\C^2$ are not collinear by \RL{U}{c}.
We define
\[
M:=\set{\{u,v\}\subset W}{u\neq v \text{ and }  u,v\in V^\m_i \text{ for some }i\in I^\m }.
\]
As a direct consequence of the definitions, we have
\[
M\subset \Lambda\subset\{\{1,2\},\{1,3\},\{2,3\}\}.
\]
Notice that if $M=\{\{1,2\},\{1,3\},\{2,3\}\}$, then $0,1,2,3\in V^\m_i$ for some $i\in I^\m$.
If $\{1,2\},\{1,3\}\in M$, then $0,1,2,3\in V_i$ for some $i\in I$,
which implies that \Wlog $\{2,3\}\in M$ by the definition of reduced graphs.
Hence, up to relabeling of vertices, exactly one of the following cases holds:
\begin{enumerate}[topsep=0pt,itemsep=0pt,leftmargin=20mm,label=\textbf{Case~\arabic*.},ref={Case~\arabic*}]
\item\label{case1} $W=\{1,2,3\}$ and $M=\varnothing$.
\item\label{case2} $W=\{1,2,3\}$ and $M=\{\{1,2\}\}$.
\item\label{case3} $W=\{1,2,3\}$ and $M=\{\{1,2\},\{1,3\},\{2,3\}\}$.
\item\label{case4} $W=\{1,2\}$ and $M\subset\{\{1,2\}\}$.
\item\label{case5} $W\subset \{1\}$ and $M=\varnothing$.
\end{enumerate}
As a direct consequence of the definitions, we have
\[
F_y\subset
\set{(r,\bt;s,\bu)\in P(G)}{(r_0,s_0)\in\Gamma
}.
\]
If $\{1,2\}\in M$, then the realizations $r, s\in\C^{2V}$ send the vertices $\{0,1,2\}$
to a pair of triangles in~$\C^2$ that are related by some transformation~$t\in T$.
Therefore,
\[
F_y\subset
\set{(r,\bt;s,\bu)\in P(G)}{(r_0,s_0)\in\Gamma''
},
\]
where
\[
\Gamma'':=
\begin{cases}
\Gamma' & \text{if } \{1,2\}\in M, \text{ and}\\
\Gamma  & \text{otherwise}.
\end{cases}
\]
If $y$ is general, then $M=\Lambda$ as a straightforward consequence of the definitions.
We prove the main assertion by applying \Cref{lem:fib,lem:codim} to each case separately.
\begin{Mlist}
\item \ref{case1}: $W=\{1,2,3\}$ and $M=\varnothing$.
In this case, $V=V^\m\cup\{0\}$ and $I=I^\m\cup\{a,b,c\}$ \st $G_a=H_{0,1}$, $G_b=H_{0,2}$, $G_c=H_{0,3}$,
where $H_{u,v}$ denotes the graph consisting of the edge~$\{u,v\}$.
It follows from \Cref{lem:dim} that
\[
\dim P(G)-\dim P(G^\m)=4(|V|-|V^\m|)+6(|I|-|I^\m|)-2(\sum_{i\in I}|V_i|-\sum_{i\in I^\m}|V^\m_i|)=10.
\]
Suppose that $(r,\bt;s,\bu)\in F_y$ so that
\[
y=(r|_{V^\m},\bt|_{I^\m};s|_{V^\m},\bu|_{I^\m})
\quad\text{and}\quad
\bt_i\circ r|_{V_i}=\bu_i\circ s|_{V_i}
\text{ for all }i\in\{a,b,c\}.
\]
For all $i\in\{a,b,c\}$, the transformation $\bu_i$ is uniquely determined by~$\bt_i$ and~$y$.
This implies that the data consisting of $(r_0,s_0;\bt_a,\bt_b,\bt_c)\in\Gamma''\times T^3$ together with~$y$
recovers uniquely the point $(r,\bt;s,\bu)$ in the fiber $F_y$.
We established that the following function is injective
\[
\varphi\c F_y\hookrightarrow\Gamma''\times T^3,\quad (r,\bt;s,\bu)\mapsto (r_0,s_0;\bt_a,\bt_b,\bt_c).
\]
First suppose that $y$ is general so that $M=\Lambda=\varnothing$.
It follows from \Cref{lem:fib} that $\Gamma''$ is irreducible and $\dim \Gamma''=1$.
Since $\dim \Gamma''\times T^3=10$ and $\dim F_y=\dim \varphi(F_y)$,
we find that $\dim F_y\leq 10$.
As $y$ is general, we have $\dim F_y\geq\dim P(G)-\dim P(G^\m)=10$
and thus
\[
\dim F_y=\dim P(G)-\dim P(G^\m)=10.
\]
This implies that $\varphi(F_y)\subset\Gamma''\times T^3$ is Zariski dense in
the irreducible set $\Gamma''\times T^3$
and thus the fiber~$F_y$ is irreducible.

Next, we assume that $\dim F_y=\dim P(G)-\dim P(G^\m)+c=10+c$ for some~$c\in \Z_{>0}$.
We may assume up to relabeling vertices that $\Lambda$
is equal to either $\varnothing$, $\{\{1,2\}\}$, $\{\{1,2\},\{2,3\}\}$ or $\{\{1,2\},\{1,3\},\{2,3\}\}$.
Hence, it follows from \Cref{lem:fib} that $\dim\Gamma''\in\{1,2\}$ (although $\Gamma''$ may be reducible).
Since $10\leq \dim \Gamma''\times T^3\leq 11$ and $c>0$,
we observe that $\dim F_y=\dim \varphi(F_y)=11$, $\dim \Gamma''=2$ and $c=1$.
Since $\dim \Gamma''=2$ it follows from \Cref{lem:fib} that $\Lambda=\{\{1,2\},\{1,3\},\{2,3\}\}$.
Notice that $\{1,3\},\{2,3\}\notin M$ and thus it follows from \Cref{lem:codim} that
\[
\dim\set{y\in P(G^\m)}{\dim F_y=11}\leq \dim P(G^\m)-2<\dim P(G^\m)-c.
\]
Therefore, the main assertion holds for this case.
\end{Mlist}
The remaining four cases are similar and left to the reader.
For \ref{case2}, we have
$\dim P(G)-\dim P(G^\m)=4$,
$\dim\Gamma''\in\{1,2\}$,
$\dim F_y=\dim\Gamma''+\dim T$,
\[
\Lambda\in\Bigl\{\{\{1,2\}\},~\{\{1,2\},\{2,3\}\},~\{\{1,2\},\{1,3\},\{2,3\}\}\Bigr\}
\]
and $\dim\set{y\in P(G^\m)}{\dim F_y=5}\leq \dim P(G^\m)-2$.
For \ref{case3}, \ref{case4} and \ref{case5},
$\dim F_y=\dim P(G)-\dim P(G^\m)$ for all $y\in P(G^\m)$.
In all cases, $F_y$ is irreducible for general $y\in P(G^\m)$.
\end{proof}

\begin{lemma}[induction step]
\label{lem:ip}
If $G=(V,E)$ is a sparse graph with marking $\m\in V^Z$
\st $P(G^\m)$ is irreducible,
then $P(G)$ is irreducible.
\end{lemma}

\begin{proof}
By \Cref{lem:cm},
there exists a comparison map $\CM\c P(G)\to P(G^\m)$,
which is by definition a morphism.
By assumption, $P(G^\m)$ is irreducible and
we know from \Cref{lem:dim} that $P(G)$ is a complex manifold and thus equidimensional.
It now follows from \Cref{lem:ifc} that the hypothesis of \Cref{lem:if} is satisfied
and thus $P(G)$ must be irreducible.
\end{proof}

\begin{proposition}
\label{prp:PG}
If $G$ is a sparse graph, then $P(G)$ is irreducible.
\end{proposition}

\begin{proof}
We prove by induction on~$|V|$ that $P(G)$ is irreducible, where $G=(V,E)$.

Induction basis: $|V|=1$.
In this case $P(G)\cong (U_V\times T)^2$ and thus $P(G)$ is irreducible.

Induction step: $|V|>1$.
By \Cref{lem:cm}, there exists a marking~$\m\in V^Z$
\st its reduced graph~$G^\m=(V^\m,E^\m)$ is sparse.
Since $|V^\m|<|V|$, we find that the fiber square~$P(G^\m)$ is irreducible by the induction hypothesis
and thus $P(G)$ must be irreducible by the induction step: \Cref{lem:ip}.
\end{proof}

\begin{proof}[Proof of \Cref{thm:main}]
By \RL{em}{c}, the sparse graph $G$ admits a unique max-tight decomposition.
It follows from \Cref{prp:PG} that $P(G)$ is irreducible
and thus the proof is concluded by \Cref{prp:main}.
\end{proof}

\begin{lemma}
\label{lem:cor}
If $G=(V,E)$ is a sparse graph with max-tight decomposition~$\{G_i\}_{i\in I}$,
then the following are equivalent:
\begin{Menum}
\item\label{lem:cor:i}
The two realizations $r,s\in\C^{2V}$ lie in the same irreducible
component~$C$ of the fiber~$\cE_G^{-1}(\lambda)$ for some $\lambda\in\C^E$.
\item\label{lem:cor:ii}
There exists $\bt,\bu\in T^I$ \st
$(r,\bt),(s,\bu)\in \C^{2V}\times T^I$
lie in the same irreducible component of a fiber
of~$\cS_G$.
\end{Menum}
\end{lemma}

\begin{proof}
\ref{lem:cor:i}$\Rightarrow$\ref{lem:cor:ii}.
Recall from \Cref{lem:dia} that the diagram at~\EQN{dia} commutes.
It follows that $\pi^{-1}(C)\subset (\mu^{-1}\circ\cP_G\circ\cS_G)^{-1}(\lambda)$
and thus $(r,\bt),(s,\bu)\in\pi^{-1}(C)$ for all $\bt,\bu\in T^I$.
Since $C$ and $T^I$ are irreducible, we conclude that $\pi^{-1}(C)$ must be an irreducible component
of a fiber of~$\cS_G$.

\ref{lem:cor:ii}$\Rightarrow$\ref{lem:cor:i}.
There exists $\lambda\in\C^E$
\st $(r,\bt)$ and~$(s,\bu)$ lie in an irreducible component~$D$
of the fiber~$(\mu^{-1}\circ\cP_G\circ\cS_G)^{-1}(\lambda)$.
Because the diagram at~\EQN{dia} commutes, we deduce that the $\lambda$-compatible realizations~$r$
and~$s$ lie in the subset~$\pi(D)$ of the fiber~$\cE_G^{-1}(\lambda)$.
Since $D$ is irreducible, we conclude that the projected component~$\pi(D)$ must be irreducible as well.
\end{proof}

\begin{proof}[Proof of \Cref{cor:main}]
\ref{cor:main:i}$\Rightarrow$\ref{cor:main:ii}.
It follows from \Cref{lem:cor} that there exists $\bt,\bu\in T^I$
\st $(r,\bt)$ and $(s,\bu)$ belong to an irreducible component of
a fiber of the subgraph map~$\cS_G$.
By definition of the subgraph map, we find that
$\bt_i\circ r|_{V_i}=\bu_i\circ s|_{V_i}$ for all~$i\in I$.
This implies that for all $i\in I$, we have $t\circ r|_{V_i}=s|_{V_i}$
with $t:=\bu_i^{-1}\circ \bt_i$.

\ref{cor:main:ii}$\Rightarrow$\ref{cor:main:i}.
The assumption that there exists~$t\in T$ \st $t\circ r|_{V_i}=s|_{V_i}$ for all~$i\in I$ is,
by the definition of the subgraph map~$\cS_G$ and generality of the edge length assignment~$\lambda$,
equivalent to the assertion that there exists $\bt,\bu\in T^I$
\st $(r,\bt)$ and $(s,\bu)$ lie in a general fiber of~$\cS_G$.
It follows from \Cref{prp:PG} that $P(G)$ is irreducible and thus
by \Cref{lem:cS} the general fiber of~$\cS_G$ is irreducible.
We now conclude from \Cref{lem:cor} that the $\lambda$-compatible realizations
$r$ and $s$ belong to the same irreducible component of the fiber~$\cE_G^{-1}(\lambda)$.

\ref{cor:main:ii}$\Leftrightarrow$\ref{cor:main:iii}.
The max-tight subgraph~$G_i$ is for all $i\in I$
minimally rigid by \Cref{thm:laman} and thus two realizations~$r|_{V_i}$ and $s|_{V_i}$
belong to the same irreducible component of~$\cE_{G_i}^{-1}(\lambda|_{E_i})$
if and only if
these realizations are equal up to composition with some transformation~$t\in T$.
\end{proof}

\section{Application to coupler curves of calligraphs}
\label{sec:app}

In this section, we consider so called ``calligraphs''
that for general edge length assignment and  up to transformations
admit a 1-dimensional set of compatible realizations.
In \Cref{cor:cal}, we characterize invariants of the irreducible components
of ``coupler curves'' of calligraphs, namely curves that are defined by realizations of some distinguished vertex.
This corollary depends on \RP{gal}{b} in \APP{scheme}.
We conclude this section with a short example
for real connected components of sets of compatible realizations.

Before we state \Cref{cor:cal}, let us first recall some definitions from \cite{2023-cal}.

A \df{calligraph} is defined as a graph $G=(V,E)$ with edge $\{1,2\}\in E$
and vertex $0\in V$ \st either $(V,E\cup\{\{1,0\}\})$ or $(V,E\cup\{\{2,0\}\})$ is tight.

The graphs $L$, $R$ and $C_v$ in \Cref{fig:cal}
are for all $v\in\Z_{\geq 3}$ calligraphs and play a special role.
The graphs in \Cref{fig:coupler,fig:main1,fig:main2} are all examples of calligraphs.

\begin{figure}[!ht]
\centering
\csep{5mm}
\begin{tabular}{ccc}
\begin{tikzpicture}[yscale=0.8]
\node[vertex] (a) at (0,0)   [label={[labelsty]below:$1$}] {};
\node[vertex] (b) at (3,0)   [label={[labelsty]below:$2$}] {};
\node[vertex] (c) at (1.5,1.5) [label={[labelsty]above:$0$}] {};
\draw[edge] (a)edge(b) (a)edge(c);
\node[] at (1.5,-0.7) {$L$};
\end{tikzpicture}
&
\begin{tikzpicture}[yscale=0.8]
\node[vertex] (a) at (0,0)   [label={[labelsty]below:$1$}] {};
\node[vertex] (b) at (3,0)   [label={[labelsty]below:$2$}] {};
\node[vertex] (c) at (1.5,1.5) [label={[labelsty]above:$0$ }] {};
\draw[edge] (a)edge(b) (b)edge(c);
\node[] at (1.5,-0.7) {$R$};
\end{tikzpicture}
&
\begin{tikzpicture}[yscale=0.8]
\node[vertex] (a) at (0,0)   [label={[labelsty]below:$1$}] {};
\node[vertex] (b) at (3,0)   [label={[labelsty]below:$2$}] {};
\node[vertex] (c) at (2,0.7) [label={[labelsty]above right:$v$}] {};
\node[vertex] (d) at (1.5,1.5) [label={[labelsty]above:$0$ }] {};
\draw[edge] (a)edge(b) (b)edge(c) (c)edge(d) (c)edge(a);
\node[] at (1.5,-0.7) {$C_v$};
\end{tikzpicture}
\end{tabular}
\caption{Three basic calligraphs, where $v\in \Z_{\geq 3}$.}
\label{fig:cal}
\end{figure}
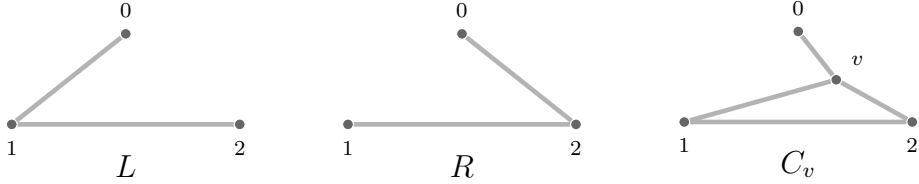

A \df{calligraphic split} of a graph~$G$ is defined as a pair~$(G_1,G_2)$ of calligraphs
$G_1=(V_1,E_1)$ and $G_2=(V_2,E_2)$
\st
\[
G=G_1\cup G_2,\quad
V_1\cap V_2=\{0,1,2\}
\quad\text{and}\quad
E_1\cap E_2=\{\{1,2\}\}.
\]
We define the bilinear map $\_\cdot\_\c\Z^3\times\Z^3 \to\Z$ as
\[
(a_1,b_1,c_1)\cdot (a_2,b_2,c_2)=2\,(a_1\,a_2-b_1\,b_2-c_1\,c_2).
\]
A \df{class} for calligraphs is a function that assigns to each calligraph~$G$
an element $[G]$ of~$\Z^3$ \st the following two axioms are fulfilled.
\begin{enumerate}[topsep=0mm,itemsep=0mm,label=A\arabic*.,ref=A\arabic*]
\item\label{axm:1}
$[L]=(1,1,0)$, $[R]=(1,0,1)$ and $[C_v]=(2,0,0)$ for all $v\in \Z_{\geq3}$.
\item\label{axm:2}
If $(G_1,G_2)$ is a calligraphic split, then $c(G_1\cup G_2)=[G_1]\cdot[G_2]$.
\end{enumerate}

By \citep[Theorem~I]{2023-cal} the class for calligraphs exists and is unique
(see \citep[Theorem~2.6]{2023-sphere} for its spherical counterpart).
Classes of calligraphs can be computed using \citep[Algorithm~1]{2023-cal}
and its software implementation~\cite{2022-m}.

\begin{example}
Let us consider the graphs $H$, $C_5$ and $G=H\cup C_5$ in \Cref{fig:split}
and observe that $(H,C_5)$ is a calligraphic split for~$G$.
We have $[H]=(6,2,2)$, $[C_5]=(2,0,0)$ and thus $c(G)=[H]\cdot [C_5]=24$ by \AXM{2}.
\END
\end{example}

\begin{figure}[!ht]
\centering
\csep{7mm}
\begin{tabular}{ccc}
\begin{tikzpicture}
\node[vertex] (a) at (0,0)   [label={[labelsty]below:$1$}] {};
\node[vertex] (b) at (3,0)   [label={[labelsty]below:$2$}] {};
\node[vertex] (c) at (1,2)   [label={[labelsty]above:$3$}] {};
\node[vertex] (d) at (2,1.5) [label={[labelsty]above:$4$}] {};
\node[vertex] (e) at (1,1)   [label={[labelsty]below:$0$}] {};
\draw[edge] (a)edge(b) (b)edge(d) (d)edge(c) (c)edge(a) (c)edge(e) (e)edge(d);
\node[] at (1.5,-0.7) {$H$};
\end{tikzpicture}
&
\begin{tikzpicture}
\node[vertex] (a) at (0,0)   [label={[labelsty]below:$1$}] {};
\node[vertex] (b) at (3,0)   [label={[labelsty]below:$2$}] {};
\node[vertex] (c) at (2,0.5) [label={[labelsty]above:$5$}] {};
\node[vertex] (d) at (1,1)   [label={[labelsty]above:$0$ }] {};
\draw[edge] (a)edge(b) (b)edge(c) (c)edge(d) (c)edge(a);
\node[] at (1.5,-0.7) {$C_5$};
\end{tikzpicture}
&
\begin{tikzpicture}
\node[vertex] (a) at (0,0)   [label={[labelsty]below:$1$}] {};
\node[vertex] (b) at (3,0)   [label={[labelsty]below:$2$}] {};
\node[vertex] (c) at (1,2)   [label={[labelsty]above:$3$}] {};
\node[vertex] (d) at (2,1.5) [label={[labelsty]above:$4$}] {};
\node[vertex] (e) at (1,1)   [label={[labelsty]below:$0$}] {};
\node[vertex] (f) at (2,0.5) [label={[labelsty]above:$5$}] {};
\draw[edge] (a)edge(b) (b)edge(d) (d)edge(c) (c)edge(a) (c)edge(e) (e)edge(d);
\draw[edge] (f)edge(a) (f)edge(b) (f)edge(e);
\node[] at (1.5,-0.7) {$G=H\cup C_5$};
\end{tikzpicture}
\end{tabular}
\caption{$(H,C_5)$ is a calligraphic split for $G$.}
\label{fig:split}
\end{figure}
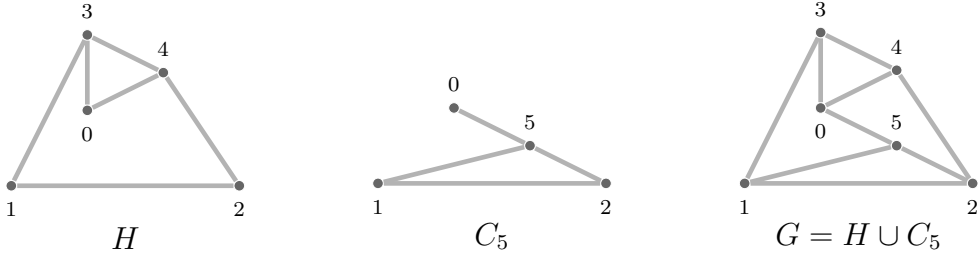

Suppose that $G=(V,E)$ is a calligraph.
Its \df{set of normed edge length assignments} is defined as
$\Omega_G:=\set{\lambda\in\C^E}{\lambda_{\{1,2\}}=1}$.
The \df{coupler curve} of $G$ \wrt normed edge length assignment~$\lambda\in\Omega_G$
is defined as
\[
\CouplerCurve(G,\lambda):=\set{r_0\in \C^2}{r\in \cE_G^{-1}(\lambda),~r_1=(0,0)\text{ and }r_2=(1,0)}.
\]
The \df{coupler multiplicity} of $G$ is defined as
\[
m(G):=|\set{r\in \cE_G^{-1}(\lambda)}{r_0=p,~r_1=(0,0)\text{ and }r_2=(1,0)}|,
\]
where both $\lambda\in\Omega_G$ and $p\in \CouplerCurve(G,\lambda)$ are general.
The coupler multiplicity is well-defined by \citep[Lemma~16]{2023-cal}
and only depends on the graph.

We call a calligraph $G=(V,E)$ \df{thin} if the following two properties are satisfied,
where $H=(V,E\cup\{\{1,0\},\{2,0\}\})$:
\begin{Mlist}
\item If $|V|>3$ and we remove any edge from the graph~$H$, then the resulting graph is tight.
\item If we remove any two of vertices from the graph~$H$,
then the resulting graph remains connected.
\end{Mlist}
We remark that the notion of ``thin'' is closely related to ``generic global rigidity''
(see \citep[Corollary~1.7]{2005} and \citep[Lemma~17]{2023-cal}).

\begin{example}
\label{exm:m}
The left calligraph in \Cref{fig:m} is thin and has coupler multiplicity $1$.
The right calligraph in \Cref{fig:m} has coupler multiplicity $2$,
since for each realization of the vertex~$0$, there are two realizations of the vertex~$3$.
Notice that the right calligraph~$G=(V,E)$ is not thin. Indeed, if
we remove the vertices~$1$ and $2$ from the graph~$(V,E\cup\{\{1,0\},\{2,0\}\})$,
then the resulting graph is not connected and
consists of the isolated vertices~$0$ and~$3$.
\END
\end{example}
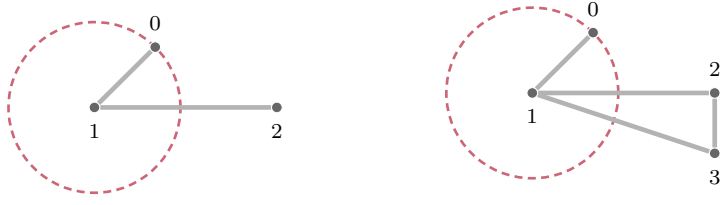
\begin{figure}[!ht]
\centering
\csep{10mm}
\begin{tabular}{cc}
\begin{tikzpicture}[scale=0.8]
\draw[traj] (0,0) circle [radius=1.414];
\node[vertex] (a) at (0,0) [label={[labelsty]below:$1$}] {};
\node[vertex] (b) at (3,0) [label={[labelsty]below:$2$}] {};
\node[vertex] (c) at (1,1) [label={[labelsty]above:$0$}] {};
\draw[edge] (a)edge(b) (a)edge(c);
\end{tikzpicture}
&
\begin{tikzpicture}[scale=0.8]
\draw[traj] (0,0) circle [radius=1.414];
\node[vertex] (a) at (0,0)  [label={[labelsty]below:$1$}] {};
\node[vertex] (b) at (3,0)  [label={[labelsty]above:$2$}] {};
\node[vertex] (c) at (1,1)  [label={[labelsty]above:$0$}] {};
\node[vertex] (d) at (3,-1) [label={[labelsty]below:$3$}] {};
\draw[edge] (a)edge(b) (a)edge(c) (a)edge(d) (b)edge(d);
\end{tikzpicture}
\end{tabular}
\caption{The left and right calligraphs have coupler multiplicities~$1$ and $2$, \resp.
The left calligraph is thin and the right calligraph is not thin.}
\label{fig:m}
\end{figure}

If $C\subset \C^2$ is an irreducible curve,
then we denote by $\Sing C$ its \df{singular locus}.
Recall from \APP{scheme} that $g(C)$ denotes the geometric genus of~$C$.
The following corollary addresses
\citep[Conjecture~2 and Open problem~3]{2023-cal}.

\begin{corollary}
\label{cor:cal}
Suppose that $G=(V,E)$ is a calligraph with class $[G]=(a,b,c)$
and max-tight decomposition $\{G_i\}_{i\in I}$, and let
\[
k:=\prod_{i\in I}c(G_i).
\]
If $G$ is thin and $\lambda\in\Omega_G$ a general edge length assignment,
then $G$ has coupler multiplicity~$m(G)=1$ and there exist $k$ different irreducible curves $C_1,\ldots,C_k\subset \C^2$
\st
\[
\CouplerCurve(G,\lambda)=C_1\cup\cdots\cup C_k
\quad\text{and}\quad
\deg C_1=\ldots=\deg C_k=2\,a/k.
\]
Moreover, if $\alpha:=(a/k,b/k,c/k)$, then for all $1\leq i,j\leq k$,
\[
g(C_i)=g(C_j)\leq \tfrac{1}{2}\,\alpha\cdot \bigl(\alpha-(2,1,1)\bigr) + 1 -|\Sing C_i|.
\]
\end{corollary}

\begin{proof}
Since $\lambda\in\Omega_G$ is general,
we know from \Cref{thm:main} that
\[
\cE_G^{-1}(\lambda)=U_1\cup\cdots\cup U_k,
\]
where $U_1,\ldots,U_k$ are pairwise different irreducible components.
Let
\[
H:=\set{r\in \C^{2V}}{r_1=(0,0),~r_2=(1,0)}
\quad\text{and}\quad
F:=\cE^{-1}_G(\lambda)\cap H.
\]
The transformation group~$T$ acts on each component~$U_i$
which implies that
\[
F=F_1\cup\cdots\cup F_k,
\]
where $F_i:=U_i\cap H$ is
an irreducible component for all $1\leq i\leq k$.
Let the projection~$\rho\c \C^{2V}\to \C^2$ send the realization~$r$ to the point~$r_0$.
We observe that
\[
\rho(F)=\CouplerCurve(G,\lambda),
\]
and thus $\rho(F_1),\ldots,\rho(F_k)$ are the irreducible components of
$\CouplerCurve(G,\lambda)$.
Because $G$ is a thin calligraph, it follows from \citep[Corollary~I(P2)]{2023-cal}
that the coupler multiplicity~$m(G)$ is equal to~$1$ (see also \citep[Lemma~16]{2023-cal}).
This implies that $\rho|_F$ is birational.
By \RP{gal}{b} in \APP{scheme} applied to the edge map~$\cE_G$ restricted to the subset~$H$ and the projection~$\rho$,
it follows that
$\deg\rho(F_i)=\deg\rho(F_j)$ and $g(\rho(F_i))=g(\rho(F_j))$ for all $1\leq i,j\leq k$.
The proof is now concluded by \citep[Corollary~I(P6)]{2023-cal}.
\end{proof}

\begin{example}
Suppose that $G$ is the graph in \Cref{fig:main1}
and recall from \Cref{exm:main1} that $G$ has max-tight decomposition
$\{G_1,\ldots,G_8\}$ \st $c(G_1)=c(G_2)=4$ and $c(G_3)=\cdots=c(G_8)=1$.
We apply \citep[Algorithm~1]{2023-cal,2022-m} to compute its class
\[
[G]=(272,0,0).
\]
We observe that $G$ is a thin calligraph and
thus it follows from \Cref{cor:cal} that for some general edge length assignment~$\lambda$
its $\CouplerCurve(G,\lambda)$ consist of $16$ irreducible components.
Each component has degree $34$ and geometric genus at most $256$,
since
$(272/16,0/16,0/16)=(17,0,0)$
and
$\tfrac{1}{2}(17,0,0)\cdot(15,-1,-1)+1=256$.
\END
\end{example}

\begin{example}
Suppose that $G$ is the graph in \Cref{fig:main2}
and recall from \Cref{exm:main2} that each of its max-tight subgraphs $G'\subset G$
is an edge so that $c(G')=1$.
We apply \citep[Algorithm~1]{2023-cal,2022-m} to compute its class
\[
[G]=(11,3,3).
\]
Since $G$ is a thin calligraph, it follows from \Cref{cor:cal} that, for some general edge length assignment~$\lambda$,
its $\CouplerCurve(G,\lambda)$ is irreducible of degree~$22$ and its geometric genus is at most~$88=\tfrac{1}{2}(11,3,3)\cdot(9,2,2)+1$.
\END
\end{example}

\begin{remark}[real connected components]
\label{rmk:real}
In this article, we count the number of irreducible components of a set of compatible realizations
over the complex numbers.
The number of real connected components remains an open problem.
One difficulty is that there does not exists a Zariski open set of
edge length assignments \st for each element in this set
the number of real components is the same.

For example, let us consider the graph in \Cref{fig:real}
and the two different edge length assignments for this graph.
In both the left and right case each realization of the vertex~$0$
lies on a circle, where we assume that vertices~$1$ and $2$
are realized as the points $(0,0)$ and $(1,0)$, \resp.
By \Cref{thm:main}, the number of complex irreducible components is equal to~$1$.
\begin{figure}[!ht]
\centering
\csep{10mm}
\begin{tabular}{cc}
\begin{tikzpicture}[scale=0.8]
\draw[traj] (0,0) circle [radius=1.414];
\draw[colR,line width=2pt] ($({1.414*cos(30)},{1.414*sin(30)})$) arc (30:110:1.414);
\draw[colB,line width=2pt] ($({1.414*cos(250)},{1.414*sin(250)})$) arc (250:330:1.414);
\node[vertex] (a) at (0,0)  [label={[labelsty]below:$1$}] {};
\node[vertex] (b) at (3,0)  [label={[labelsty]below:$2$}] {};
\node[vertex] (c) at (1,1)  [label={[labelsty]above:$0$}] {};
\node[vertex] (d) at (3.5,0.5) [label={[labelsty]right:$3$}] {};
\draw[edge] (a)edge(b) (a)edge(c) (b)edge(d) (c)edge(d);
\end{tikzpicture}
&
\begin{tikzpicture}[scale=0.8]
\draw[colR,line width=2pt] (0,0) circle [radius=1.414];
\node[vertex] (a) at (0,0)  [label={[labelsty]below:$1$}] {};
\node[vertex] (b) at (3,0)  [label={[labelsty]below:$2$}] {};
\node[vertex] (c) at (1,1)  [label={[labelsty]above:$0$}] {};
\node[vertex] (d) at (5,1.5) [label={[labelsty]right:$3$}] {};
\draw[edge] (a)edge(b) (a)edge(c) (b)edge(d) (c)edge(d);
\end{tikzpicture}
\end{tabular}
\caption{The set of compatible realizations for the left and right edge length assignments
consists of two and one real connected component, \resp.}
\label{fig:real}
\end{figure}
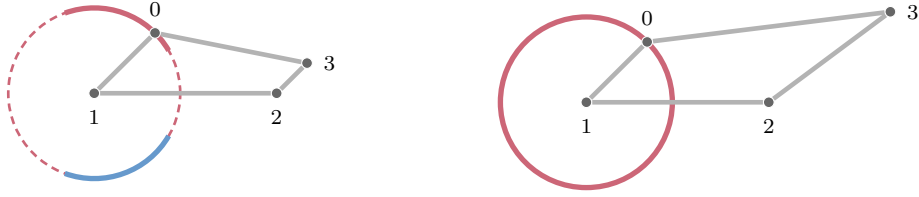

Each realization of the vertex~$0$
corresponds to two realizations of the graph as we can reflect vertex~$3$
along the line containing the realizations of the vertices~$0$ and~$2$.
However, for the left case in \Cref{fig:real}
the vertex~$3$ admits two real realizations if and only if the realization of the vertex~$0$
lies on either the red or blue circular arc.
Hence, the complex irreducible component contains two real connected components.
Moreover, the set of edge length assigments \st the number of real connected components is $2$
is Zariski dense in the set of all edge length assignments.

If we increase the length of the edges~$\{0,3\}$ and $\{2,3\}$ as in right case of \Cref{fig:real},
then the number of real connected components equals $1$.
Notice that the set of edge length assigments \st there is one real connected component
is Zariski dense as well.

We refer to the classical ``Grashof conditions for four bar linkages'' for a more detailed analysis
of the graphs in \Cref{fig:real}.
\END
\end{remark}

\section{On the higher dimensional setting}
\label{sec:dim}

In this section we present a counter example for a naive 3-dimensional generalization of
our main result \Cref{thm:main}. We characterize graphs for which we expect
that our methods also work in the higher dimensional setting, namely
graphs that admit a ``max-rigid decomposition''.

We can define an edge map $(\C^n)^V\to\C^E$ of a graph $G=(V,E)$ for any dimension~$n\geq 2$
and thus one may ask for the number~$c(G)$ of irreducible components
of the general fiber of this edge map when $n\geq 3$.
We may assume that this edge map is dominant and thus
$G$ is rigid if the general fiber is finite.
However, if $n\geq 3$, then tightness is not equivalent to rigidity.
To resolve this issue, let us call a subgraph $G'\subset G$ \df{max-rigid} if it is rigid and there exists
no rigid subgraph~$G''$ \st $G'\subsetneq G''\subset G$.
We call $\{G_i\}_{i\in I}=\{(V_i,E_i)\}_{i\in I}$ a \df{max-rigid decomposition} for~$G$
if $G_i\subset G$ is a max-rigid subgraph for all $i\in I$,
$V=\cup_{i\in I}V_i$, $E=\cup_{i\in I} E_i$
and $E_i\cap E_j=\varnothing$ for all different~$i,j\in I$.

If $G$ admits a max-rigid decomposition, then we believe that our methods
can be adapted to the higher dimensional setting.
However, the problem is that in the higher dimensional setting not every graph admits a max-rigid decomposition.

For example, suppose that $n=3$ and consider the graph~$G$ as defined in
\Cref{fig:3d}. Its max-rigid subgraphs are $G_1$, $G_2$ and $G_3$,
where $G_1$ is the triangle with vertices $\{1,2,3\}$,
$G_2$ is the triangle with vertices $\{2,3,4\}$,
and $G_3$ is the subgraph of~$G$ spanned by the vertices $\{1,4,5,6,7\}$.
Notice that the triangles $G_1$ and $G_2$ have a common edge~$\{2,3\}$,
and thus $\{G_1,G_2,G_3\}$ is \emph{not} a max-rigid decomposition of~$G$.

\begin{figure}[!ht]
\centering
\begin{tikzpicture}[scale=0.4]
\node[vertex] (1) at (-7,0)   [label={[labelsty]left:$1$}] {};
\node[vertex] (2) at (-4,5)  [label={[labelsty]left:$2$}] {};
\node[vertex] (3) at ( 4,5)  [label={[labelsty]right:$3$}] {};
\node[vertex] (4) at ( 7,0)   [label={[labelsty]right:$4$}] {};
\node[vertex] (5) at (-0.5,-2)   [label={[labelsty]above:$5$}] {};
\node[vertex] (6) at ( 1, 0.5)  [label={[labelsty]above:$6$}] {};
\node[vertex] (7) at ( 2,-4)  [label={[labelsty]below:$7$}] {};
\draw[edge] (1)edge(2) (2)edge(3) (3)edge(4) (2)edge(4) (1)edge(3)
            (1)edge(5) (1)edge(6) (1)edge(7)
            (6)edge(4) (5)edge(4) (7)edge(4)
            (5)edge(6) (6)edge(7) (7)edge(5);
\end{tikzpicture}
\caption{A graph~$G$ \st for a general choice of edge lengths
its set of realizations into 3-dimensional complex space
consists of $c(G)=8$ irreducible components.
The number of realizations of its three
max-rigid subgraphs are $1$, $1$ and $4$, \resp.}
\label{fig:3d}
\end{figure}

Now let us consider the realization of $G$ into $\R^3\subset\C^3$ as is illustrated in \Cref{fig:3d}.
We denote by $H_{uvw}\subset \R^3$ the plane spanned by the points corresponding to the vertices~$\{u,v,w\}$.
We may assume up to rotations and translations that the coordinates for the
vertices $\{5,6,7\}$ are the same for each realization of~$G$.
Since $G_1$ and $G_2$ are triangles, we have
\[
c(G_1)=c(G_2)=1.
\]
We obtain different realizations of $G_3$ by reflecting the vertices $1$ and/or $4$ about the plane~$H_{567}$,
and thus
\[
c(G_3)=2\cdot 2=4.
\]
Let us now determine $c(G)$.
We observe that the max-rigid subgraph~$G_3$ can rotate about the line spanned
by the vertices $1$ and $4$.
In addition to reflecting vertices~$1$ and~$4$
about the plane~$H_{567}$, it is also possible to either
\begin{Mlist}
\item reflect vertex $2$ about the plane $H_{134}$, or reflect vertex $3$ about the plane $H_{124}$.
\end{Mlist}
The latter two reflections result in two realizations
of $G$ that are related via translations and rotations.
Hence, we find that
\[
c(G)=c(G_1)\cdot c(G_2)\cdot c(G_3)\cdot 2=8.
\]
Notice that if we replace the subgraph~$G_3$ in a realization of $G$ with an edge, then the edges of the
resulting graph coincide with the edges of a tetrahedron.
The additional factor two in the factorization of $c(G)$ corresponds to different orientations of this tetrahedron.

The generalization of \Cref{thm:main} to dimension $n\geq 3$ remains an open problem.

\appendix
\addappheadtotoc

\section{The fiber product lemma}
\label{sec:fiber}

The goal of this appendix is to prove \Cref{lem:XxX}.
For this we require scheme theoretic language, but
we provide precise references for the non-expert.

We assume that varieties are affine, but not necessarily irreducible.
If $X$ is a variety,
then $\C[X]$ and $\C(X)$ denote the \df{function ring}
and \df{function field} of~$Y$, respectively
(see \citep[pages~4 and 16]{1977} where function ring is called ``coordinate ring'').
The \df{generic fiber} of a morphism $f\c X\to Y$ between irreducible varieties $X$ and $Y$
is defined as
\[
\Spec \C[X]\otimes_{\C[Y]}\C(Y).
\]
We refer to \citep[pages~xvi, 16, 87 and 89, and Exercise~II.2.7 at page~80]{1977}
or \citep[Definition~1.5 in Chapter~3]{2002L}
for more details about the fiber at a ``generic point'',
which is closely related to our notion of general point.

\begin{lemma}
\label{lem:ring}
Suppose that $f\c X\to Y$ is a dominant morphism between irreducible varieties.
\begin{claims}
\item\label{lem:ring:a}
The fiber product $X\times_f X$ is irreducible
if and only if
the tensor product
\[
\C[X]\otimes_{\C[Y]}\C[X]
\]
is an integral domain.
\item\label{lem:ring:b}
The fiber $f^{-1}(y)$ is irreducible for general~$y\in Y$
if and only if
\[
\C[X]\otimes_{\C[Y]}\C(Y)
\]
is an integral domain.
\end{claims}
\end{lemma}

\begin{proof}
Assertions~\ref{lem:ring:a} and~\ref{lem:ring:b} are a direct consequence of the
definitions.
\end{proof}

\begin{lemma}
\label{lem:tensor}
If $R$ is an $A$-module and $S\subset A$ a multiplicatively closed subset,
then
\[
R\otimes_A S^{-1}A\cong S^{-1}R.
\]
\end{lemma}

\begin{proof}
See for example \citep[Proposition~3.5]{1969} for this well-known result.
\end{proof}

\begin{lemma}
\label{lem:step1}
If $\C[X]\otimes_{\C[Y]}\C[X]$ is an integral domain for varieties $X$ and $Y$,
then $\C(X)\otimes_{\C(Y)}\C(X)$
is an integral domain as well.
\end{lemma}

\begin{proof}
If $R:=\C[X]\otimes_{\C[Y]}\C[X]$, then its localization at the multiplicatively closed set~$S=R\setminus\{0\}$
is as a direct consequence of the definitions at \citep[Ch.~3]{1969} equal to
\[
S^{-1}R=\C(X)\otimes_{\C(Y)}\C(X).
\]
Since $R$ is an integral domain, $S^{-1}R$ is an integral domain as well.
\end{proof}

\begin{lemma}
\label{lem:step2}
If $\E/\F$ is a field extension \st
$\E\otimes_\F \E$ an integral domain,
and $e\in \E$ is algebraic over~$\F$, then $e\in \F$.
\end{lemma}

\begin{proof}
Suppose by contradiction that $e\notin \F$.
Let $p(t)=\sum^d_{i=1}p_i\cdot t^i$
be the minimal polynomial of~$e$
so that $d\geq 2$ and $p_d=1$.
Since $p(e)=0$, we have
\[
0=p(e)\otimes 1-1\otimes p(e)=\sum^d_{i=1}p_i\cdot\left(e^i\otimes 1-1\otimes e^i\right).
\]
We observe that for all $1\leq i\leq d$,
\[
x^i-y^i=(x-y)\cdot\sum_{j=0}^{i-1}x^jy^{i-1-j},
\]
and thus
\[
e^i\otimes 1-1\otimes e^i=(e\otimes 1-1\otimes e)\cdot \sum_{j=1}^{i-1} e^j\otimes e^{i-1-j}.
\]
We establised that $\alpha\cdot\beta=0$, where
\[
\alpha:=e\otimes 1-1\otimes e
\qquad\text{and}\qquad
\beta:=\sum^d_{i=1}\sum_{j=1}^{i-1}p_i\cdot e^j\otimes e^{i-1-j}.
\]
Since the minimal polynomial~$p(t)$ is of degree~$d$, we find that
the following set is linear independent:
\[
\set{e^i\otimes e^j}{1\leq i,j\leq d-1}.
\]
This implies that $\beta\neq 0$,
and thus $\alpha$ and $\beta$ are zero divisors in $\E\otimes_\F \E$.
We arrived at a contradiction as $\E\otimes_\F \E$ is an integral domain.
\end{proof}

\begin{lemma}
\label{lem:step3}
Suppose that $f\c X\to Y$ is a dominant morphism between varieties
and let $f^\star\c \C(Y)\to \C(X)$ be the homomorphism between function fields
that sends the function~$\alpha$ to the composition~$\alpha\circ f$.
Let $\E:=\C(X)$ and $\F:=f^\star(\C(Y))$ so that $\E/\F$ is a field extension.
If $\F$ is algebraically closed in $\E$,
then $\C[X]\otimes_{\C[Y]}\C(Y)$ is an integral domain.
\end{lemma}

\begin{proof}
Suppose by contradiction that
$R:=\C[X]\otimes_{\C[Y]}\C(Y)$ is not an integral domain.
In this case, the generic fiber of~$f$ decomposes into $k\geq 2$ irreducible components:
\[
\Spec R=C_1\cup\cdots\cup C_k.
\]
By Sard's theorem the generic fiber is smooth and thus its irreducible components coincide
with the connected components.
Let $r\in R$ define the function $\Spec R\to \C$
that sends $c$ to $i$ if $c\in C_i$ for all $1\leq i\leq k$
and consider the polynomial
\[
\tilde{p}:=(t-1)\cdots(t-k).
\]
Notice that the composition $\tilde{p}\circ r$ vanishes identically.
Moreover, we observe that $\tilde{p}\in \C(Y)[t]$ is the minimal polynomial of $r$.
With the multiplicative set~$S:=\C[Y]\setminus\{0\}$, we find that
\[
R=\C[X]\otimes_{\C[Y]}S^{-1}\C[Y].
\]
We know from \Cref{lem:tensor} that $R\cong S^{-1}\C[X]$.
Since $S^{-1}\C[X]\subset \C(X)=\E$, there exists an injective ring homomorphism
\[
\varphi\c R \to \E,
\]
\st $\varphi(\C(Y))=\F$.
It follows that $\varphi(r)\in\E$ has a nonzero minimal polynomial in~$\F[t]$.
As $\F$ is algebraically closed in~$\E$ by assumption, we established that $\varphi(r)\in \F$.
This implies that $\varphi(r)$ defines a rational function~$X\to \C$
that is the composition of $f\c X\to Y$ with some function $\alpha\c Y\to\C$.
We arrived at a contradiction as $(\alpha\circ f)(C_i)=(\alpha\circ f)(C_j)$
although $r(C_i)\neq r(C_j)$
for all $1\leq i<j\leq k$.
\end{proof}

\begin{proof}[Proof of \Cref{lem:XxX}.]
It follows from \citep[Proposition~III.10.4]{1977}
that $f\c X\to Y$ is a smooth morphism between irreducible affine varieties.

\ref{lem:XxX:a} This assertion follows from \citep[Proposition~III.10.1(d)]{1977}.

\ref{lem:XxX:b} Since $X\times_f X$ is irreducible,
we know from \RL{ring}{a} that $\C[X]\otimes_{\C[Y]}\C[X]$ is an integral domain.
Hence, it follows from \Cref{lem:step1} that $\C(X)\otimes_{\C(Y)}\C(X)$ is an integral domain as well.
We apply \Cref{lem:step2} with $\E=\C(X)$ and $\F=f^\star(\C(Y))$,
and find that $\F$ is algebraically closed in $\E$.
Therefore, $\C[X]\otimes_{\C[Y]}\C(Y)$ is an integral domain by \Cref{lem:step3}.
We conclude from \RL{ring}{b} that the general fiber of~$f$ is irreducible.
\end{proof}

\section{Invariants of components of fibers}
\label{sec:scheme}

The goal of this appendix is to prove \Cref{prp:gal,cor:inv}, which concerns
algebro geometric invariants of irreducible components of sets
of realizations that are compatible with some general edge length assignment.
We assume the notation of \APP{fiber}.

If $X\subset \C^n$ is a variety, then its \df{dimension},
\df{degree} and \df{arithmetic genus} are denoted by $\dim X$, $\deg X$ and $p_a(X)$,
\resp~(see \citep[\textsection I.7 and Exercise~I.7.2]{1977}).
If $X$ is a curve, then we denote by $g(X)$ its \df{geometric genus} \citep[Remark~II.8.18.2]{1977}.

\begin{lemma}
\label{lem:gf}
If $f\c X\to Y$ is a dominant morphism between irreducible varieties,
then $\C[X]\otimes_{\C[Y]}\C(Y)$ is an integral domain.
\end{lemma}

\begin{proof}
Notice that $\C[X]$ is a $\C[Y]$-module
via the pullback $f^*\c\C[Y]\to\C[X]$.
By \Cref{lem:tensor}, $\C[X]\otimes_{\C[Y]}\C(Y)\cong S^{-1}\C[X]$,
where $S:=\C[Y]\setminus\{0\}$ is a multiplicatively closed set.
Since $\C[X]$ is an integral domain, its localization is integral as well.
\end{proof}

\begin{proposition}
\label{prp:gal}
Suppose that $f\c X\subset\C^n\to Y$ is a dominant morphism between smooth and irreducible varieties~$X$ and~$Y$.
Let $C$ and $C'$ be different components of the fiber~$f^{-1}(y)$
for some general~$y\in Y$.
\begin{claims}
\item\label{prp:gal:a}
$C$ is smooth, $C\cap C'=\varnothing$, $\dim C=\dim C'$, $\deg C=\deg C'$ and $p_a(C)=p_a(C')$.
\item\label{prp:gal:b}
If $\dim C=1$, $\rho\c \C^n\to \C^m$ is the projection to the first $m\leq n$ coordinates and its restriction~$\rho|_{f^{-1}(y)}$
is birational, then
\[
\deg \rho(C)=\deg\rho(C')
\quad\text{and}\quad
g(\rho(C))=g(\rho(C')).
\]
\end{claims}
\end{proposition}

\begin{proof}
\ref{prp:gal:a}
Since $f$ is a morphism between differentiable manifolds,
it follows from Sard's theorem that the general fiber of $f$ is smooth,
which implies that $C$ is smooth and $C\cap C'=\varnothing$.
It follows from \Cref{lem:gf} that the generic fiber
\[
Z:=\Spec \C[X]\otimes_{\C[Y]}\C(Y)
\]
is irreducible over the function field $\C(Y)$.
Suppose that $\E/\C(Y)$ is an extension field \st $Z=Z_1\cup Z_2\cup\cdots\cup Z_k$
decomposes into $k\geq 2$ irreducible components,
where $Z_i$ is defined by polynomials with coefficients in~$\E$ for all $1\leq i\leq k$.
By \citep[Proposition~I.7.6 at page~52 and Exercise~I.7.2 at page~54]{1977}
there exists for all $1\leq i\leq k$ a unique \df{Hilbert polynomial}~$P_{Z_i}\in \C(Y)[t]$
with \df{leading coefficient}~$\lc(P_{Z_i})\in \Z_{>0}$
\st
$\dim Z_i=\deg P_{Z_i}$,
$\deg Z_i=(\deg P_{Z_i})!\cdot\lc(P_{Z_i})$
and
$
p_a(Z_i)=(-1)^{\deg P_{Z_i}}\cdot(P_{Z_i}(0)-1).
$
The Galois group of the field extension~$\E/\C(Z)$
induces a field isomorphism~$\C(Z_i)\to \C(Z_j)$ for all $1\leq i<j\leq k$.
It follows from the procedure for computing Hilbert polynomials
as described at~\citep[\textsection9.3]{2007}
that this field isomorphism identifies $P_{Z_i}$ with $P_{Z_j}$
so that
$\deg P_{Z_i}=\deg P_{Z_j}$, $\lc(P_{Z_i})=\lc(P_{Z_j})$
and $P_{Z_i}(0)=P_{Z_j}(0)$.
The proof is now concluded as a direct consequence of the definitions
of generic and general fibers.

\ref{prp:gal:b}
The ideal of~$\rho(Z)$ is obtained from the ideal of~$Z$ by
elimination of variables (see \citep[\textsection3.2,Theorem~3]{2007}).
Since $Z$ is irreducible over the function field~$\C(Y)$,
we deduce that $\rho(Z)$ is irreducible over~$\C(Y)$ as well.
It follows from \citep[Theorem~II.8.19, Remark~III.7.12.2 and Exercise~IV.1.8]{1977}
that $g(\rho(C))=p_a(C)$.
The same arguments as before conclude the proof.
\end{proof}

The following simple example is to clarify
\Cref{prp:gal} and its proof.

\begin{example}
\label{exm:irr}
Suppose that $f\c \C^3\to \C^2$ maps $(x,y,z)$ to $(x^2,z)$.
The generic fiber is defined as $\Spec\F[x,y,z]/\langle x^2-s,z-t\rangle$,
where $\F:=\set{p/q}{p,q\in\C[s,t],q\neq 0}$ denotes the
function field of~$\C^2$.
Although $x^2-s$ is irreducible over the function field~$\F$, it factors as $(x-u)(x+u)$
over the extension field $\E:=\F(u)/\langle u^2-s\rangle$.
The Galois group of this field extension $\E/\F$ is generated by a field automorphism
that sends $u$ to $-u$ and leaves~$\F$ elementwise invariant.
If the projection $\rho\c \C^3\to \C^2$ maps $(x,y,z)$ to $(x,y)$,
then the image of the generic fiber of $f$ is $\Spec\F[x,y]/\langle x^2-s\rangle$.
Hence, the image of the generic fiber is irreducible over the field~$\F$.
This implies that the irreducible components $\Spec\E[x,y]/\langle x-u\rangle$
and $\Spec\E[x,y]/\langle x+u\rangle$
have the same dimension, degree and arithmetic genus.
\END
\end{example}

\begin{corollary}
\label{cor:inv}
Suppose that $G=(V,E)$ is a sparse graph, $\lambda\in\C^E$ a general edge length assignment,
and let $C,C'$ be different irreducible components in the set~$\cE_G^{-1}(\lambda)\subset\C^{2V}$
of $\lambda$-compatible realizations.
Then $C$ is smooth,
\[
C\cap C'=\varnothing,\quad
\dim C=\dim C',\quad
\deg C=\deg C'
\quad\text{and}\quad
p_a(C)=p_a(C').
\]
\end{corollary}

\begin{proof}
The edge map~$\cE_G\c \C^{2V}\to \C^E$ is dominant by \RL{em}{a}
and thus the proof is concluded by applying \RP{gal}{a} to~$\cE_G$ with $\C^{2V}\cong\C^{2|V|}$.
\end{proof}

\section*{Acknowledgements}
\addcontentsline{toc}{section}{Acknowledgements}

Financial support for Niels Lubbes and Mehdi Makhul was provided by the Austrian Science Fund (FWF): P36689 and P33003.

\bibliography{comp}

\textbf{Addresses of authors:}
\\Institute for Algebra, Research Institute for Symbolic Computation (RISC),
\\Johannes Kepler University, Linz, Austria
\\\url{info@nielslubbes.com}
\\\url{josef.schicho@risc.jku.at}
\\[2mm]
Johann Radon Institute for Computational and Applied Mathematics (RICAM),
\\Austrian Academy of Sciences, Linz, Austria
\\\url{mehdi.makhul@oeaw.ac.at}
\\\url{audie.warren@oeaw.ac.at}
\end{document}